\documentclass[letterpaper, reqno,11pt]{article}
\usepackage[margin=1.0in]{geometry}
\usepackage{color,latexsym,amsmath,amssymb, amsthm, graphicx, bbm}

\numberwithin{equation}{section}

\newcommand{\RR}{\mathbb{R}}
\newcommand{\CC}{\mathbb{C}}

\newcommand{\ZZ}{\mathbb{Z}}

\newcommand{\eps}{\varepsilon}

\newcommand{\vol}{\lambda}

\newtheorem{thm}{Theorem}[section]
\newtheorem{lem}[thm]{Lemma}
\newtheorem{clm}[thm]{Claim}

\newtheorem{cor}[thm]{Corollary}

\newtheorem*{main-discV2Thm}{Theorem \ref{main-discV2}}
\newtheorem*{curvatureProjectionsThm}{Theorem \ref{curvatureProjections}}
\newtheorem*{dimension-expanderThm}{Theorem \ref{dimension-expander}}
\newtheorem*{main-entropy-growthThm}{Theorem \ref{main-entropy-growth}}
\newtheorem*{main-energy-dispersionThm}{Theorem \ref{main-energy-dispersion}}

\theoremstyle{remark}
\newtheorem{rem}[thm]{Remark}
\newtheorem{example}{Example}
\newtheorem{defn}[thm]{Definition}

\begin{document}
\title{Dimension-expanding polynomials and the discretized Elekes-R\'onyai theorem}
\author{Orit E. Raz, Joshua Zahl}
\maketitle

\begin{abstract}
We characterize when bivariate real analytic functions are ``dimension expanding'' when applied to a Cartesian product. If $P$ is a bivariate real analytic function that is not locally of the form $P(x,y) = h(a(x) + b(y))$, then whenever $A$ and $B$ are Borel subsets of $\mathbb{R}$ with Hausdorff dimension $0<\alpha<1$, we have that $P(A,B)$ has Hausdorff dimension at least $\alpha + \eps$ for some $\eps(\alpha)>0$ that is independent of $P$. The result is sharp, in the sense that no estimate of this form can hold if $P(x,y) = h(a(x) + b(y))$. We also prove a more technical single-scale version of this result, which is an analogue of the Elekes-R\'onyai theorem in the setting of the Katz-Tao discretized ring conjecture. As an application, we show that a discretized non-concentrated set cannot have small nonlinear projection under three distinct analytic projection functions, provided that the corresponding 3-web has non-vanishing Blaschke curvature. 
\end{abstract}

\section{Introduction}
In \cite{ErdSze}, Erd\H{o}s and Szemer\'edi proved that if $A\subset\RR$ is a finite set, then either the sum set $A+A=\{a+a^\prime\colon a,a^\prime\in A\}$ or the product set $A.A=\{aa^\prime\colon a,a^\prime\in A\}$ must have cardinality much larger than that of $A$. 
\begin{thm}[Erd\H{o}s-Szemer\'edi]
There are positive constants $\eps$ and $c$ so that for all finite sets $A\subset\RR$, we have
\begin{equation}\label{sumProdIneq}
\#(A+A) + \#(A.A) \geq c(\#A)^{1+\eps}.
\end{equation}
\end{thm}
They conjectured that \eqref{sumProdIneq} should hold for all $\eps<1$. This conjecture remains open, and the current best result in this direction is due to Rudnev and Stevens \cite{RS20}, who showed that \eqref{sumProdIneq} is true whenever $\eps< 1/3 + 2/1167$. 

The Erd\H{o}s-Szemer\'edi Theorem quantifies the principle that a subset of $\RR$ cannot be approximately closed under both addition and multiplication. In \cite{ER}, Elekes and R\'onyai developed this idea in a slightly different direction, and proved the following:
\begin{thm}[Elekes-R\'onyai]\label{ElekesRonyai}
Let $P$ be a bivariate real polynomial. Then either $P$ is one of the special forms $P(x,y) = h(a(x)+b(y))$ or $P(x,y) = h(a(x)b(y))$, where $h,a,b$ are univariate real polynomials, or 
that for all finite sets of real numbers $A,B$ of cardinality $N$, we have 
\begin{equation}\label{growthPoly}
\#P(A,B) =\omega(N).
\end{equation}
\end{thm}
If $P(x,y) = h(a(x)+b(y))$ or $P(x,y) = h(a(x)b(y))$, where $h,a,b$ are univariate real polynomials, then we call $P$ a (polynomial) special form. If $P$ is not a (polynomial) special form then we call it an expanding polynomial. This dichotomy has since been generalized in several directions. 
A slightly more general formulation of Theorem~\ref{ElekesRonyai} asserts that for finite sets $A,B,C$ of real numbers, each of cardinality $N$,  the number of points of $A\times B\times C$ that lie on the graph of $P$, that is, on the zero set of the polynomial $F(x,y,z)=z-P(x,y)$, is $o(N^2)$. 
In \cite{ES}, Elekes and Szab\'o proved a generalization of Theorem~\ref{ElekesRonyai} by showing that if $F$ is any trivariate polynomial, then either $Z(F)$ must have a special structure or 
\begin{equation}\label{smallIntersectionZPABC}
\#\big( Z(F)\cap (A \times B \times C)\big)\leq cN^{2-\eps}
\end{equation} 
for all finite sets of real (or complex) numbers $A,B,C$ of cardinality $N$, where here $c>0$ depends on the degree of $F$. In \cite{RSS} and \cite{RSZ}, the first author along with Sharir, Solymosi, and De Zeeuw obtained quantitative improvements in the expansion exponent $\eps$ in \eqref{smallIntersectionZPABC}, and an analogous improvement in \eqref{growthPoly}. A similar result was independently proved by Wang in \cite{W}. Very recently, Makhul,  Roche-Newton, Stevens, and Warren \cite{MRSW} showed that if we assume the uniformity conjecture, then it is possible to further improve the expansion exponent $\eps$ in \eqref{smallIntersectionZPABC} if $A,B,C$ are sets of rational numbers and if the function  $F$ belongs to a particular family of polynomials. In \cite{RS}, Shem-Tov and the first author generalized a quantitative version of Theorem~\ref{ElekesRonyai} to $d$-variate polynomials, while in \cite{BB}, Bays and Breuillard  generalized Elekes and Szab\'o's result (for a small $\eps>0$) to arbitrary varieties in $\CC^d$. 

All of the results discussed so far have concerned cardinality estimates for finite subsets of $\RR$ or $\CC$. 
Similar questions can be asked for metric entropy. In this direction, Katz and Tao proposed the discretized ring conjecture \cite{KT}, which was solved by Bourgain \cite{B03}. The discretized ring conjecture (now a theorem) is similar to the Erd\H{o}s-Szemer\'edi theorem, except cardinality has been replaced by metric entropy---if $(X,d)$ is a metric space and $\delta>0$, we define the ``$\delta$-covering number'' $\mathcal{E}_{\delta}(X)$ to be the minimum number of balls of radius $\delta$ required to cover $X$.
\begin{thm}[Bourgain]\label{discRingThm}
For each $0<\alpha<1$, there is a number $\eps = \eps(\alpha)>0$ and $s = s(\alpha)>0$ so that the following holds for all $\delta>0$ sufficiently small. Let $A\subset [1,2]$ be a set with $\mathcal{E}_{\delta}(A) = \delta^{-\alpha}$. Suppose that $A$ satisfies the following non-concentration condition for each interval $J$ of length at least $\delta$.
\begin{equation}\label{nonConcentrationDiscRingConj}
\mathcal{E}_{\delta}(A \cap J) \leq \delta^{-s} |J|^{\alpha}\mathcal{E}_{\delta}(A).
\end{equation}
Then
\begin{equation*}
 \mathcal{E}_{\delta}(A+A) +  \mathcal{E}_{\delta}(A.A) \geq \mathcal{E}_{\delta}(A)^{1+\eps}.
\end{equation*}
 
\end{thm}
The non-concentration condition \eqref{nonConcentrationDiscRingConj} arises naturally when discretizing fractal sets; see \cite{KT} or Section \ref{dimExpansionSec} for details. In particular, Bourgain used Theorem \ref{nonConcentrationDiscRingConj} to prove the following ``dimension expansion'' result for subsets of $\RR$:
\begin{cor}[Bourgain]\label{growthInRings}
For each $0<\alpha<1$ there is a number $\eps=\eps(\alpha)>0$ so that the following holds. Let $A\subset\RR$ be a Borel set with $\dim A = \alpha$. Then 
\begin{equation*}
\max\big(\dim (A+A),\ \dim(A.A)\big)\geq\alpha+\eps.
\end{equation*}
\end{cor}
Corollary~\ref{growthInRings} solves a quantitative version of the Erd\H{o}s-Volkmann ring conjecture \cite{EV} (proved by Edgar and Miller \cite{EM} slightly earlier), which asserts that there does not exist a proper Borel subring of the reals with positive Hausdorff dimension. 

Theorem \ref{discRingThm} has seen a number of extensions and generalizations. In \cite{BG}, Bourgain and Gamburd proved a variant of \ref{discRingThm} with a less restrictive version of the non-concentration condition \eqref{nonConcentrationDiscRingConj}, while in \cite{GKZ} Guth, Katz, and the second author found a simple new proof of Theorem \ref{discRingThm} that yields an explicit lower bound on the exponent $\eps$. Theorem \ref{discRingThm} has also been generalized to $SU(2)$ \cite{BG} as well as other settings \cite{BG2, H, HS, S}.

\subsection{New results: Expanding polynomials and analytic functions}
Our main result is a discretized version of the Elekes-R\'onyai theorem for analytic functions, in the spirit of Theorem \ref{discRingThm}. Before stating our result, we will define what it means for an analytic function to be a special form (cf. \cite[Lemma 10]{RaSh}). 
\begin{defn}\label{defnAnalyticSpecialForm}
Let $P(x,y)$ be analytic on a connected open set $U\subset\RR^2$. We say $P$ is an (analytic) special form if on each connected region of $U \backslash \big(\{ P_x = 0\}\cup \{P_y = 0\}\big)$, there are univariate real analytic functions $h$, $a$, and $b$ so that $P(x,y) = h(a(x)+b(y))$.
\end{defn}
Note that every polynomial special form is also an analytic special form; this is obvious if $P(x,y) = h(a(x) + b(y))$ for $h,a,b$ polynomial. If $P(x,y) = h(a(x)b(y))$, then on each open connected region of $\RR^2\backslash \big(\{ P_x = 0\}\cup\{P_y = 0\}\big)$ we have $a(x)b(y)$ is strictly positive or strictly negative, and thus for an appropriate choice of sign we can write $P(x,y) = h(\pm \exp(\log |a(x)| +  \log |b(y)|))$. 

We are now ready to state our main result. 

\begin{thm}\label{main-entropy-growth}
Let $0< \kappa\leq \alpha< 1$, let $U\subset\RR^2$ be a connected open set that contains $[0,1]^2$, and let $P\colon U \to\RR$ be analytic (resp.~polynomial). Then either $P$ is an analytic (resp.~polynomial) special form, or there exists $\eps = \eps(\alpha,\kappa)>0$ and $\eta = \eta(\alpha,\kappa,P)>0$ so that the following is true for all $\delta>0$ sufficiently small.

Let $A,B\subset[0,1]$ be sets with $\mathcal{E}_{\delta}(A)\geq\delta^{-\alpha}$, $\mathcal{E}_{\delta}(B)\geq\delta^{-\alpha},$ and suppose that for all intervals $J$ of length at least $\delta$, $A$ and $B$ satisfy the non-concentration conditions
\begin{equation}\label{nonConcentrationCond2}
\begin{split}
\mathcal{E}_{\delta}(A \cap J) & \le  |J|^\kappa\delta^{-\alpha-\eta},\\
\mathcal{E}_{\delta}(B \cap J) & \le  |J|^\kappa\delta^{-\alpha-\eta}.
\end{split}
\end{equation}
Then we have the entropy growth estimate
\begin{equation}\label{entropyGrowth}
\mathcal{E}_{\delta}(P(A,B))\geq \delta^{-\alpha-\eps}.
\end{equation}
If $P$ is a polynomial, then the quantity $\eta$ can be taken to depend only on $\alpha,\kappa$, and $\deg P$.
\end{thm}

\begin{rem}
Note that the entropy growth exponent $\eps$ in Theorem \ref{main-entropy-growth} is independent of $P$, and in particular independent of the degree of $P$ when $P$ is a polynomial. This is interesting for the following reasons. First, if $P$ is a polynomial then it might be possible (though likely difficult) to prove an analogue of Theorem \ref{main-entropy-growth} using Bourgain's discretized sum-product theorem (Theorem \ref{discRingThm}) and tools from Ruzsa calculus. For example, these ideas were used in \cite{BKT} to prove expansion over $\mathbb{F}_p$ for a certain explicitly specified polynomial.  However, the entropy growth exponent $\eps$ arising from such a strategy would necessarily depend on the polynomial $P$ (the degree of $P$, the number of monomial terms, etc).

Second, there is a general intuition in additive combinatorics that states that one should expect a quantity such as $P(A,B)$ to be large when there are few solutions to $P(x,y) = P(x^\prime,y^\prime)$, and conversely. Concretely, if $A$ and $B$ have cardinality $\delta^{-\alpha}$, then the statements ``there are at about $\delta^{-3\alpha+\eps}$ solutions to $P(x,y) = P(x^\prime,y^\prime)$'' and ``$P(A,B)$ has cardinality about $\delta^{-\alpha-\eps}$'' are often morally equivalent. In the present setting, however, this intuition is not entirely correct---there exists a sequence of polynomials $P_1,P_2,\ldots$ of increasing degree so that there are $\delta^{-3\alpha+o(1)}$ solutions to $P_j(x,y) = P_j(x^\prime,y^\prime)$, and yet the $\delta$-covering number of $P_j(A,B)$ remains essentially unchanged (see Example~\ref{epsDependsDForEnergyDispersion} in Section \ref{examplesSec}). This will be discussed further in Theorem \ref{main-energy-dispersion}. 
\end{rem}

We will also prove the following ``dimension expansion'' version of Theorem \ref{main-entropy-growth}, which is an analogue of Corollary \ref{growthInRings}. 
\begin{thm}\label{dimension-expander}
Let  $U\subset\RR^2$  be a connected open set that contains $[0,1]^2$ and let $P\colon U \to\RR$ be analytic (resp.~polynomial). Then either $P$ is an analytic (resp.~polynomial) special form, or for every $0< \alpha< 1$, there exists $\eps=\eps(\alpha)>0$ so that the following holds. Let $A,B\subset [0,1]$ be Borel sets with Hausdorff dimension at least $\alpha$. Then 
\begin{equation}\label{dimensionGrowth}
\dim P(A,B) \ge \alpha+\eps.
\end{equation}  
\end{thm}

Finally, we will prove a slightly more technical ``energy dispersion'' estimate, which is closely related to the $L^2$ flattening lemma that was used by Bourgain and Gamburd \cite{BG2, BG} to establish expansion in the Cayley graph of $SL_2(\mathbb{F}_p)$ and the existence of a spectral gap in certain free subgroups of $SU(2)$; see \cite[\S1.4]{TaoBook} for a nice introduction to this topic.
\begin{thm}\label{main-energy-dispersion}
Let $0< \kappa\leq \alpha< 1$, let $U\subset\RR^2$ be a connected open set that contains $[0,1]^2$, and let $P\colon U \to\RR$ be analytic (resp.~polynomial). Then either $P$ is an analytic (resp.~polynomial) special form, or there exists $\eps = \eps(\alpha,\kappa,P)>0$ and $\eta = \eta(\alpha,\kappa,P)>0$ so that the following is true for all $\delta>0$ sufficiently small.

Let $A,B\subset[0,1]$ and suppose that for all intervals $J$ of length at least $\delta$, $A$ and $B$ satisfy the non-concentration condition
\begin{equation}\label{nonConcentrationCond}
\begin{split}
\mathcal{E}_{\delta}(A \cap J) &\le  |J|^\kappa\delta^{-\alpha-\eta},\\
\mathcal{E}_{\delta}(B \cap J) &\le  |J|^\kappa\delta^{-\alpha-\eta}.
\end{split}
\end{equation}
Then we have the energy dispersion estimate
\begin{equation}\label{energyDispersionEqn}
\mathcal{E}_{\delta}\big(\{(x,x',y,y')\in A^2\times B^2 \colon P(x,y) = P(x^\prime,y^\prime) \}\big)\leq \delta^{-3\alpha+\eps}.
\end{equation}
If $P$ is a polynomial, then the quantities $\eps$ and $\eta$ can be taken to depend only on $\alpha,\kappa$, and $\deg P$.
\end{thm}

\begin{rem}
Theorems \ref{main-entropy-growth} and \ref{main-energy-dispersion} involve several parameters, and the dependencies between these parameters is somewhat subtle. In Section \ref{examplesSec} below, we will construct several examples highlighting why some of these dependencies are necessary. In particular, note that in Theorem \ref{main-energy-dispersion} the value of $\eps$ in the energy dispersion estimate \eqref{energyDispersionEqn} depends on $P$ (and in particular depends on the degree of $P$ when $P$ is a polynomial). Example \ref{epsDependsDForEnergyDispersion} in Section \ref{examplesSec} shows that this dependence is unavoidable. It is thus rather remarkable that the exponent $\eps$ in the entropy growth estimate \eqref{entropyGrowth} from Theorem \ref{main-entropy-growth} is independent of $P$. 
\end{rem}

\subsection{Expanding functions, nonlinear projections, and the Blaschke curvature form}
In \cite{ER}, Elekes and R\'onyai observed that a polynomial $P(x,y)$ is a special form if and only if $\partial_{xy}\log|P_x/P_y|$ vanishes identically. This quantity is closely related to the Blaschke curvature form, which we will now define. 
\begin{defn}
Let $U\subset\RR^2$ be a connected open set, and let $\phi_1,\phi_2,\phi_3\colon U\to\RR$ be smooth functions whose gradients are non-vanishing and pairwise linearly independent at each point of $U$. Then the Blaschke curvature form is given by 

\begin{equation}\label{BlaschkeCurvatureForm}
2 \frac{\partial}{\partial \phi_1}\frac{\partial}{\partial \phi_2} \log \frac{\partial \phi_3 / \partial\phi_1}{\partial \phi_3 / \partial\phi_2}d\phi_1\wedge d\phi_2.
\end{equation}
\end{defn}
Note that if $\phi_1(x,y) = x$, $\phi_2(x,y) = y$ and $\phi_3(x,y) = P(x,y)$, then the Blaschke curvature form is given by $(2\partial_{xy}\log|P_x/P_y|)dx\wedge dy$. It was shown in \cite{JMR} that in this setup, the Blaschke curvature form vanishes if and only if $P_x/P_y$ can be written (locally) as a product $a(x)b(y)$.

The level sets of the functions $\phi_i$ define three families of smooth curves $\mathcal{F}_1,\mathcal{F}_2,\mathcal{F}_3$ on $U$. For each point $p\in U$, there is a unique curve from each $\mathcal{F}_i$ that contains $p$, and curves from different families intersect transversely. A triple $(\mathcal{F}_1,\mathcal{F}_2,\mathcal{F}_3)$ of families of plane curves with these properties is called a (planar) 3-web. Blaschke \cite{Bl} (see also \cite{Izo} for the exact formulation given above) showed that the curvature form \eqref{BlaschkeCurvatureForm} depends only in the 3-web, and not on the choice of $\phi_1,\phi_2,\phi_3$. In particular, we may always make a (local) change of coordinates so that $\phi_1(x,y) = x$ and $\phi_2(x,y) = y$. 

\begin{thm}\label{curvatureProjections}
Let $0< \alpha< 1$, let $U\subset\RR^2$ be a connected open set, and let $K\subset U$ be compact. Let $\phi_1,\phi_2,\phi_3\colon U\to\RR$ be analytic functions whose gradients are pairwise linearly independent at each point of $U$. Then either the Blaschke curvature form \eqref{BlaschkeCurvatureForm} vanishes identically on $U$, or there exists $\eps = \eps(\alpha)>0$ and $\eta = \eta(\alpha,\phi_1,\phi_2,\phi_3)>0$ so that the following is true for all $\delta>0$ sufficiently small.

Let $X\subset K$ be a set with $\mathcal{E}_{\delta}(X)\geq \delta^{-2\alpha}$, and suppose that for all balls $B$ of diameter $r\geq\delta$, $X$ satisfies the non-concentration condition
\begin{equation}\label{nonConcentrationCondX}
\begin{split}
\mathcal{E}_{\delta}(X \cap B) &\le  r^\alpha\delta^{-2\alpha-\eta}.
\end{split}
\end{equation}
Then for at least one index $i\in \{1,2,3\}$ we have
\begin{equation}\label{atLeastOneBigProjection}
\mathcal{E}_{\delta}(\phi_i(X))\geq \delta^{-\alpha-\eps}.
\end{equation}
\end{thm}

Theorem \ref{curvatureProjections} is a discretized projection theorem in the spirit of \cite{B10}. In \cite{B10}, Bourgain proved that if $X\subset[0,1]^2$ is a set satisfying \eqref{nonConcentrationCondX}, and if $\Theta$ is a large (and non-concentrated) set of directions, then there must exist a direction $\theta\in\Theta$ so that the associated projection $\phi_{\theta}(x,y) = x\cos\theta + y\sin\theta$ satisfies \eqref{atLeastOneBigProjection}. The set $\Theta$ must be large, since if we define $X=[0,1]^2 \cap (\delta^{\alpha}\ZZ)^2$, then whenever $s$ and $t$ are integers with $|s|$ and $|t|$ small, the projection $(x,y)\cdot (s,t)$ will have small cardinality and thus fail to satisfy \eqref{atLeastOneBigProjection}. 

Note that if $\theta_1,\theta_2,\theta_3$ are directions, then the Blaschke curvature form of the triple of linear projections $\phi_i = \phi_{\theta_i}$ vanishes identically. Theorem \ref{curvatureProjections} says that the situation for non-vanishing Blaschke curvature is very different: while any set $X$ satisfying \eqref{nonConcentrationCondX} can have small projection under two projection maps $\phi_1$ and $\phi_2$ (indeed, we could define $X = \phi_1^{-1}(X_1) \cap \phi_2^{-1}(X_2)$, where $X_1$ and $X_2$ are non-concentrated sets of cardinality $\delta^{-\alpha}$), it is impossible to have small projection under three projection maps. 

As a motivating example, if $p_1,p_2,p_3\in\RR^2$ are not collinear and if we define $\phi_i(q) = |p-q|$, then Elekes and Szab\'o \cite[Theorem 32]{ES} showed that the Blaschke curvature form of $(\phi_1,\phi_2,\phi_3)$ does not vanish identically on $\RR^2$. This leads to the following variant of the pinned Falconer distance problem.
\begin{cor}\label{pinnedDistanceCor}
Let $0< \alpha< 1$ and let $p_1,p_2,p_3$ be three points that are not collinear. Let $K\subset\RR^2$ be a compact set that does not intersect any of the three lines $\ell_{p_ip_j}$. Then there exists $\eps = \eps(\alpha)>0$ and $\eta = \eta(\alpha,p_1,p_2,p_3)>0$ so that the following is true for all $\delta>0$ sufficiently small.

Let $X\subset K$ be a set with $\mathcal{E}_{\delta}(X)\geq \delta^{-2\alpha}$, and suppose that for all balls $B$ of diameter $r\geq\delta$, $X$ satisfies the non-concentration condition
\begin{equation}\label{nonConcentrationCondDist}
\begin{split}
\mathcal{E}_{\delta}(X \cap B) &\le  r^\alpha\delta^{-2\alpha-\eta}.
\end{split}
\end{equation}
Then for at least one index $i\in \{1,2,3\}$ we have
\begin{equation}\label{atLeastOneBigPinnedDistance}
\mathcal{E}_{\delta}( \{ |q-p_i| \colon q \in X\})\geq \delta^{-\alpha-\eps}.
\end{equation}
\end{cor}
\begin{rem}
The requirement that none of the three lines $\ell_{p_ip_j}$ intersect $K$ is imposed because the line $\ell_{p_ip_j}$ is precisely the set where the gradients of the functions $q\mapsto |q-p_i|$ and $q\mapsto |q-p_j|$ are linearly dependent.
\end{rem}

\begin{rem}
We conjecture that the quantity $\eta$ can be taken independent of the choice of $p_1,p_2,p_3$.
\end{rem}

\begin{rem}
Theorems \ref{main-entropy-growth}, \ref{dimension-expander}, and \ref{main-energy-dispersion} can also be re-phrased in the language of triples $\phi_1,\phi_2,\phi_3$ with non-vanishing Blaschke curvature. Results in this spirit have appeared in the harmonic analysis literature; see \cite{Chr} and the references therein for further discussion. 
\end{rem}

\subsection{Proof ideas and structure of the paper}

Our proof of Theorems \ref{main-entropy-growth} and \ref{main-energy-dispersion} follow the major ideas from the proof of the Elekes-R\'onyai theorem in \cite{RSZ}. In \cite{RSZ}, the authors reduced a growth estimate of the type \eqref{entropyGrowth} to an energy dispersion estimate of the type \eqref{energyDispersionEqn}, and then reduced the latter to an incidence problem involving points and curves in the plane. The authors then showed that either this collection of points and curves satisfied the hypotheses of a generalized Szemer\'edi-Trotter-type theorem, or a certain semi-algebraic set related to the polynomial $P$ had large dimension. In the former case the the desired energy dispersion estimate is true, while in the latter case, the polynomial $P$ is a special form.

We will follow a similar strategy, though the discretized setting presents new difficulties. First, in the discretized setting it is not possible to directly reduce an entropy growth estimate of the type \eqref{entropyGrowth} to an energy dispersion estimate of the type \eqref{energyDispersionEqn}, since the $\eps$ gain in Theorem \ref{main-energy-dispersion} depends on the degree of the polynomial $P$ (and the examples in Section \ref{examplesSec} show that this dependence is unavoidable), while the $\eps$ gain in Theorem \ref{main-entropy-growth} is independent of $D$; this phenomena is not present in the discrete setting. For example, the polynomial $P(x,y) = x + y + (x^2 + y^2)^{D/2}$, where $D$ is a large even integer, highlights some of the difficulties. When $|x|,|y|\lesssim\delta^{1/D}$ then at scale $\delta$, $P(x,y)$ is indistinguishable from the special form $Q(x,y)=x+y$; this is problematic because the region where $|x|,|y|\lesssim\delta^{1/D}$ can contain fairly large portion of $A\times B$.  To tackle this difficulty, we will introduce an auxiliary function $K_P(x,y)$, which is small on regions where $P$ closely resembles a special form. 

When $P$ is not a special form, the authors in \cite{RSZ} reduced the problem of counting quadruples $P(x,y) = P(x',y')$ to a point-curve incidence problem, and they then applied a variant of the Szemer\'edi-Trotter theorem. Our approach is similar in spirit; in regions where our auxiliary function $K_P(x,y)$ is large, we reduce the problem of counting (discretized) quadruples $P(x,y) = P(x',y')$ to a nonlinear version of Bourgain's discretized projection theorem \cite{B10}, which was recently proved by Shmerkin \cite{Shm}.

The structure of the paper is as follows. In Section \ref{analFunctNonConcSec} we will introduce tools that describe how ``fractal'' objects such as products of non-concentrated sets can interact with thin neighborhoods of the vanishing locus of an analytic function.  In Section \ref{ShmerkinSection} we will introduce Shmerkin's nonlinear discretized projection theorem, and we will recast this theorem into a (rather technical) statement about energy dispersion. In section \ref{auxiliaryFunctionSection} we will introduce the auxiliary function $K_P$ mentioned above, and we show that it has two key properties. First, if $K_P$ vanishes identically, then $P$ is a special form. Second, when $K_P$ is large, then an energy dispersion estimate of the form \eqref{energyDispersionEqn} holds with an exponent $\eps$ that is independent of $D$. In Sections \ref{mainThmProofsSectionOne} and \ref{mainThmProofsSectionTwo}, we prove Theorems \ref{main-entropy-growth} and \ref{main-energy-dispersion}, respectively. To prove Theorem \ref{main-entropy-growth}, we use the tools from Section \ref{analFunctNonConcSec} to find a large square where $K_P$ is large and then apply the results from Section \ref{auxiliaryFunctionSection}. To prove Theorem \ref{main-energy-dispersion} we need to work slightly harder---in addition to finding a large square where $K_P$ is large, we also need to analyze the region where $K_P$ is small. In Section \ref{dimExpansionSec} we use the results from Section \ref{auxiliaryFunctionSection} and standard techniques from geometric measure theory to prove Theorem \ref{dimension-expander}. Finally,  in Section \ref{BlaschkeSection} we will prove Theorem \ref{curvatureProjections}.

\subsection{Examples}\label{examplesSec}
In this section, we will construct examples demonstrating that some of the parameter dependencies described in Theorems \ref{main-entropy-growth} and \ref{main-energy-dispersion} are necessary.

First, we will construct an example showing that if $P$ is a special form, then no nontrivial entropy expansion, dimension expansion, or energy dispersion estimates can hold. 
\begin{example}[special forms are not expanders]
Let $P=h(a(x)+b(y))$, where $h,a,b$ are univariate analytic functions. Let $I\subset[0,1]$ and $J\subset[0,1]$ be intervals of length $\sim 1$ so that $|\partial_x a|\sim 1$ on $I$ and $|\partial_x b|\sim 1$ on $J$. Let $I^\prime = a(I)$ and $J^\prime =b(J)$. Let $\Lambda\subset\RR$ be an arithmetic progression consisting of $\delta^{-\alpha}$ points with spacing $\delta^{\alpha} \min(|I^\prime|, |J^\prime|)$, and let $A = a^{-1}(\Lambda + v),\ B = b^{-1}(\Lambda+w)$, where $v,w\in\RR$ are chosen so that $\Lambda+v\subset I^\prime$ and $\Lambda+w\subset J^\prime$. Then $A$ and $B$ satisfy the hypotheses of Theorem \ref{main-entropy-growth}, but
\begin{equation*}
\mathcal{E}_{\delta}(P(A,B)) = \mathcal{E}_{\delta} (h(\Lambda + \Lambda + (v+w))\lesssim 2\delta^{-\alpha}.
\end{equation*}
Similarly,
\begin{equation*}
\begin{split}
\mathcal{E}_{\delta}\big(& \{(x,x',y,y')\in A^2\times B^2 \colon P(x,y) = P(x^\prime,y^\prime) \}\big)\\
& \sim \mathcal{E}_{\delta}\big( \{(x,x',y,y')\in (\Lambda+v)^2\times (\Lambda+w)^2 \colon h(x+y) = h(x'+y') \}\big)\\
& \geq \mathcal{E}_{\delta}\big( \{(x,x',y,y')\in \Lambda^4 \colon x+y = x'+y' \}\big)\\
& \gtrsim \delta^{-3\alpha}.
\end{split}
\end{equation*}

In particular, if $P$ is the special form $h(a(x)+b(y))$, then there are sets $A$ and $B$ so that the entropy growth estimate \eqref{entropyGrowth} and the energy dispersion estimate \eqref{energyDispersionEqn} will not hold for any $\eps>0$. A similar example can be constructed for Theorem \ref{dimension-expander} by choosing $\Lambda\subset\RR$ to be compact set with $\dim(\Lambda) = \dim(\Lambda+\Lambda) = \alpha$; see i.e.~\cite{Ko}.
\end{example}

For the next examples, let $D\geq 4$ be an even integer, $0<c\le 1$, and let $0\leq \alpha,\eta\leq 1$. Define
\begin{equation}\label{defAB}
A=B=\{[j\delta^{\alpha+\eta},\ j\delta^{\alpha+\eta}+\delta] \colon j=1,\ldots,\delta^{-\alpha}\},
\end{equation} 
and define
\begin{equation}\label{defP}
P(x,y)  = x+y + c(x^2 + y^2)^{D/2}. 
\end{equation}
Note that $P$ is not a special form, as can be easily verified (e.g. using the characterization in Lemma~\ref{special}). If $\alpha+\eta\leq 1$ then $\mathcal{E}_{\delta}(A)=\mathcal{E}_{\delta}(B)=\delta^{-\alpha}$, and $A$ and $B$ satisfy the non-concentration conditions \eqref{nonConcentrationCond2} and \eqref{nonConcentrationCond}. The polynomial $P$ is not a special form, and it satisfies $|\partial_x P|\sim 1$ and $|\partial_y P|\sim 1$ on $[0,1]^2$.

\begin{example}[$\eps$ must depend on $\alpha$] 
Let $A,B$ and $P$ be as defined in \eqref{defAB} and \eqref{defP}, with $D=4$, $\kappa=\eta=0,$ and $c=1$. Then $\mathcal{E}_{\delta}(P(A,B))\lesssim \mathcal{E}_{\delta}(A)\mathcal{E}_{\delta}(B)=\delta^{-2\alpha}$ and $\mathcal{E}_{\delta}(P(A,B))\lesssim\delta^{-1}$. We immediately see that in Theorem \ref{main-entropy-growth} we must have $\eps\leq\min(\alpha,1-\alpha)$. Since $|\nabla P|\gtrsim 1$ on $[0,1]^2$, a standard double-counting argument and Cauchy-Schwarz shows that the same constraint on $\eps$ must also hold for Theorem \ref{main-energy-dispersion}.
\end{example}

\begin{example}[$\eta$ must depend on $D$]
Let $A,B$ and $P$ be as defined in \eqref{defAB} and \eqref{defP}. Note that whenever $x\in A$ and $y\in B$ with $x,y\leq \frac{1}{2}\delta^{1/D},$ we have $0\leq (x^2 + y^2)^{D/2} \leq \delta/2$. 

If $D>1/\eta$, then we have $(x^2 + y^2)^{D/2} \leq \delta/2$ for all $x\in A$ and $y\in B$, and it is easy to verify that
\begin{equation}\label{smallEntropy}
\mathcal{E}_{\delta}(P(A,B))\sim\delta^{-\alpha}.
\end{equation}
We conclude that the quantity $\eta$ from Theorem \ref{main-entropy-growth} must depend on $D$. As in the previous example, \eqref{smallEntropy} can be combined with a standard double-counting argument and Cauchy-Schwarz to show that the quantity $\eta$ from Theorem \ref{main-entropy-growth} must depend on $D$.
\end{example}

\begin{example}[$\eps$ must depend on $D$ in Theorem \ref{main-energy-dispersion}]\label{epsDependsDForEnergyDispersion}
Let $A,B$ and $P$ be as defined in \eqref{defAB} and \eqref{defP}, with $\eta = 0$. Then 
\begin{equation*}
\begin{split}
\mathcal{E}_{\delta}\big( & \{(x,x',y,y')\in A^2\times B^2\colon P(x,y) = P(x^\prime,y^\prime) \}\big)\\
& \geq \mathcal{E}_{\delta}\big( \{(x,x',y,y')\in (A^2\times B^2) \cap [0, \frac{1}{4}\delta^{1/D}]^4 \colon P(x,y) = P(x^\prime,y^\prime) \}\big)\\
& \sim \mathcal{E}_{\delta}\big( (A^2\times B)\cap [0, \frac{1}{4}\delta^{1/D}]^3\big)\\
&\sim \delta^{-3\alpha + 3/D}.
\end{split}
\end{equation*}
We conclude that the quantity $\eps$ in the bound \eqref{energyDispersionEqn} from Theorem \ref{main-energy-dispersion} must depend on $D$. Note that by Theorem~\ref{main-entropy-growth} we have $\mathcal{E}_{\delta}(P(A,B))\gtrsim\delta^{-\alpha-\eps}$, where $\eps>0$ is independent of $D$.
\end{example}

\begin{example}[$\delta$ must be sufficiently small depending on $P$]
Let $P$ be as defined in \eqref{defP}, with $D=4$. Letting $c>0$ be small, we see that $P(x,y) = x + y + O(c)$, and thus no bound of the form \eqref{entropyGrowth} or \eqref{energyDispersionEqn} can hold for $\eps>0$ unless $\delta>0$ is sufficiently small compared to $c$.
\end{example}


\subsection{Notation}
In what follows, $\delta$ will denote a small positive number. We will be interested in the asymptotic behaviour various quantities as $\delta\searrow 0$. If $f$ and $g$ are functions, we write $f\lesssim g$ (or $f = O(g)$ or $g = \Omega(f)$) if there is a constant $C$ (independent of $\delta$) so that $f\leq Cg$. Such a constant will be called an ``implicit constant in the $\lesssim$ notation.'' If $f\lesssim g$ and $g\lesssim f$, we write $f\sim g$. 

While Theorems \ref{main-entropy-growth}, \ref{dimension-expander}, and \ref{main-energy-dispersion} involve bivariate polynomials, some of our intermediate results will be slightly more general and will involve functions with domain $\RR^d$. All implicit constants in the $\lesssim$ notation are allowed to depend on $d,\alpha$ and $\kappa$. 

If $X\subset\RR^d$, we will use $\mathcal{E}_{\delta}(X)$ to denote the $\delta$-covering number of $X$; $\#X$ to denote the cardinality of $X$; and $|X|$ to denote the Lebesgue measure of $X$. To improve clarity, we sometimes use $\vol_d(X)$ in place of $|X|$ to emphasize the ambient dimension. Finally, if $t>0$ we use $N_t(X)$ to denote the $t$-neighborhood of $X$.

\subsection{Thanks}
The authors would like to thank Pablo Shmerkin for numerous helpful comments and suggestions. The authors would like to thank Michael Christ for stimulating discussions, and for alerting them to the Blaschke curvature form and its connection to special forms.

\section{Analytic functions and non-concentration}\label{analFunctNonConcSec}
In this section we will prove several technical results showing how products of non-concentrated sets can intersect thin neighborhoods of the vanishing locus of an analytic function. The main tools we will use are \L{}ojasiewicz's inequality for real analytic functions, \L{}ojasiewicz's stratification of semianalytic sets, and Whitney's decomposition of a domain into Whitney cubes. We will recall these results below.

\begin{thm}[\L{}ojasiewicz's inequality \cite{Lo}, \cite{Sol}]\label{Lojasiewicz}
Let $U\subset\RR^d$ be open, let $\phi\colon U \to\RR$ be real analytic and let $K\subset\RR^d$ be compact. Then there exist constants $c>0$ (small) and $C>0$ (large) so that for each $z\in K$, we have
\begin{equation*}
|\phi(z)| \geq c\cdot  {\rm dist}(z, Z(\phi))^{C}.
\end{equation*}
If $f$ is a polynomial, the constant $C$ can be chosen to depend only on $d$ and the degree of $f$. 
\end{thm}

\begin{thm}[Stratification of the vanishing locus of an analytic function \cite{Lo2}]\label{stratificationRealAnalSet}
Let $U\subset\RR^d$ be an open set that contains $[0,1]^d$ and let $\phi\colon U \to\RR$ be real analytic and not identically zero. Then there is a finite set $\mathcal{M}$ of smooth (proper) submanifolds of $\RR^d$ so that $Z(\phi)\cap [0,1]^d \subset \bigcup_{M\in\mathcal{M}}M$. 
\end{thm}

We will use the following version of Whitney's decomposition of open subsets of $\RR^d$, which  can be found in \cite[Appendix J]{Gra}.
\begin{thm}[Whitney]\label{whitneyCubeDecomp}
Let $\Omega$ be an open non-empty proper subset of $\RR^d$. Then there exists a countable family $\mathcal{Q}$ of closed cubes such that
\begin{itemize}
\item $\bigcup_{Q\in\mathcal{Q}} Q = \Omega$, and the cubes in $\mathcal{Q}$ have disjoint interiors.
\item For each cube $Q\in\mathcal{Q}$, we have $2Q\not\subset\Omega$, where $2Q$ denotes the 2-fold dilate of $Q$. 
\end{itemize}
\end{thm}

\subsection{Thin neighborhoods of $Z(\phi)$ and non-concentration}
In this section we will show that thin neighborhoods of the vanishing locus of an analytic function have small intersection with non-concentrated sets. 
\begin{lem}
\label{neighborhoodAnalyticVariety}
Let $M$ be a smooth (proper) submanifold of $\RR^d$, and suppose $M \cap [0,1]^d$ is compact. Then there is a constant $C>0$ so that the following holds. 

Let $A_1,\ldots,A_d\subset[0,1]$ be sets and let $A = A_1\times\cdots\times A_d$. Let $\delta>0$ and suppose that for all intervals $J$ of length at least $\delta$, each set $A_i$ satisfies the non-concentration condition
\begin{equation*}
\mathcal{E}_{\delta}(A_i \cap J) \leq \delta^{-\eta}|J|^{\kappa}\mathcal{E}_{\delta}(A_i).
\end{equation*}
Then
\begin{equation}\label{deltaNbhdEstimate}
\mathcal{E}_{\delta}\big(  A\cap N_{\delta}(M) \big) \leq C \mathcal{E}_{\delta}(A) / \min_{i}\mathcal{E}_{\delta}(A_i).
\end{equation}
Furthermore, for each $s>\delta$ we have
\begin{equation}\label{thickenedNbhdEstimate}
\mathcal{E}_{\delta}\big(  A\cap N_{s}(M) \big) \leq C \delta^{-\eta} s^{\kappa}\mathcal{E}_{\delta}(A).
\end{equation}
\end{lem}
\begin{proof}
For each $p\in M\cap [0,1]^d$, let $i(p)\in \{1,\ldots,d\}$ be an index so that the coordinate unit vector $e_{i(p)}$ is not contained in the tangent space $T_p(M)$. If $t>0$ is selected sufficiently small, then the set $M_p = M\cap B(p,t)$ is a smooth manifold, and there exists a constant $C_p$ so that for all $q,q'\in M_p$, we have
\begin{equation}\label{eiCoordinateBd}
|e_{i(p)}\cdot q - e_{i(p)}\cdot q'|\leq  C_p|\pi_{p}(q)-\pi_{p}(q')|,
\end{equation}
where $\pi_p\colon\RR^{d}\to\RR^{d-1}$ is the orthogonal projection to the $(d-1)$-dimensional subspace $e_{i(p)}^\perp\subset\RR^d$. 
Indeed, the above inequality follows from the fact that $\inf_{v\in T_p(M)}\angle(e_{i(p)},v)>0$, and the map $q\mapsto T_q(M)$ is continuous. 

Since $M\cap [0,1]^d$ is compact, we can cover $M\cap [0,1]^d$ by a finite set $\mathcal{M}$ of manifolds of the form $M_p$. It suffices to prove \eqref{deltaNbhdEstimate}  and \eqref{thickenedNbhdEstimate} for each of these manifolds. Let $C_0=\max_{M_p\in\mathcal{M}} C_p$. 

Let $M_p\in\mathcal{M}$. Let $s\geq\delta$ and let $\mathcal{Q}$ be the set of cubes of side-length $s$ aligned with the grid $(s\ZZ)^d$ that intersect $M_p$. We claim that at most $C_1$ cubes from $\mathcal{Q}$ can have the same projection under $\pi_{p}$, for some $C_1$ that depends only on $d$ and $C_0$. Indeed, if $C_1>2$ cubes have the same projection under $e_{i(p)}$, then select points $q,q'\in M_p$ from two cubes that have greatest distance. Since there are at least $C_1-2$ cubes in-between the cubes containing $q$ and $q'$, we have $|e_{i(p)}\cdot q - e_{i(p)}\cdot q'|\geq (C_1-2)s$. On the other hand, since all of the cubes (and in particular, the cubes containing $q$ and $q'$) have the same projection under $\pi_p$, we have that $|\pi_{p}(q)-\pi_{p}(q')|\leq \sqrt{d-1}s$. Thus by \eqref{eiCoordinateBd}, we must have $(C_1-2)s \leq C_0\sqrt{d-1}s,$ so $C_1\leq C_0\sqrt{d-1}+2$, which proves our claim. 
When $s=\delta$, this immediately implies \eqref{deltaNbhdEstimate}.

It remains to prove \eqref{thickenedNbhdEstimate}. For each such cube $Q\in\mathcal{Q}$ we have the estimate 
\begin{equation}\label{coveringNumberSCube}
\mathcal{E}_{\delta}(Q\cap A) \leq \delta^{-\eta}s^{\kappa}\mathcal{E}_{\delta}(A_{i(p)}) \mathcal{E}_{\delta}( \pi_{p}(A\cap Q)).
\end{equation}
Summing \eqref{coveringNumberSCube} over all cubes in $\mathcal{Q}$ and using the fact that at most $C_1$ cubes have the same projection under $\pi_p$, we obtain
\begin{equation*}
\begin{split}
\mathcal{E}_{\delta}(N_s(M)\cap A)&\leq C_1 \delta^{-\eta}s^{\kappa}\mathcal{E}_{\delta}(A_{i(p)})\mathcal{E}_{\delta}( \pi_p(A))\\
&\leq C_1 \delta^{-\eta}s^{\kappa}\mathcal{E}_{\delta}(A).\qedhere
\end{split}
\end{equation*}
\end{proof}

\begin{rem}
Note that in the special case where where exists an index $i$ so that $e_i$ is not contained in $T_p(M)$, for every $p\in M$, it is sufficient in Lemma~\ref{neighborhoodAnalyticVariety} to require the non-concentration condition only for $A_i$.
\end{rem}

Combining Lemma~\ref{neighborhoodAnalyticVariety} and Theorem~\ref{stratificationRealAnalSet}, we obtain the following.

\begin{cor}\label{coveringNbhAnalyticSet}
Let $U\subset\RR^d$ be an open set that contains $[0,1]^d$ and let $\phi\colon U \to\RR$ be real analytic and not identically zero. Then there is a constant $C$ so that the following holds. 

Let $A_1,\ldots,A_d\subset[0,1]$ be sets and let $A = A_1\times\cdots\times A_d$. Let $\delta>0$ and suppose that for all intervals $J$ of length at least $\delta$, each set $A_i$ satisfies the non-concentration condition
\begin{equation*}
\mathcal{E}_{\delta}(A_i \cap J) \leq \delta^{-\eta}|J|^{\kappa}\mathcal{E}_{\delta}(A_i).
\end{equation*}
Then
\begin{equation}
\mathcal{E}_{\delta}\big(  A\cap N_{\delta}(Z(\phi)) \big) \leq C \mathcal{E}_{\delta}(A) / \min_i \mathcal{E}_{\delta}(A_i),
\end{equation}
and for each $s>\delta$ we have
\begin{equation}\label{CoveringThinNbhdZPhi}
\mathcal{E}_{\delta}\big(  A\cap N_{s}(Z(\phi)) \big) \leq C \delta^{-\eta} s^{\kappa}\mathcal{E}_{\delta}(A).
\end{equation}
\end{cor}

Corollary~\ref{coveringNbhAnalyticSet} applied to $\phi(z)- t$, where $\phi$ is real analytic and $t\in[0,1]$, implies that for each $t$ there is a constant $C$ so that \eqref{CoveringThinNbhdZPhi} holds with this choice of $C$. The next result says that for certain carefully chosen values of $t$, 
the constant $C$ can be taken independent of $t$. 

\begin{lem}\label{uniformBdPhiT}
Let $U\subset\RR^d$ be an open set that contains $[0,1]^d$ and let $\phi\colon U \to\RR$ be real analytic and not identically zero. Then there is a constant $C$ so that the following holds. 

Let $A_1,\ldots,A_d\subset[0,1]$ be sets and let $A = A_1\times\cdots\times A_d$. Let $\delta>0$ and suppose that for all intervals $J$ of length at least $\delta$, each set $A_i$ satisfies the non-concentration condition
\begin{equation*}
\mathcal{E}_{\delta}(A_i \cap J) \leq \delta^{-\eta}|J|^{\kappa}\mathcal{E}_{\delta}(A_i).
\end{equation*}

Let $s>\delta$ and let $s^{\kappa/2}<t_0\leq 1/2$. Then there is a number $t\in [t_0,2t_0]$ so that
\begin{equation}\label{goodChoiceOfT}
\mathcal{E}_{\delta}\big(  A\cap N_{s}(\{\phi = t\}) \big) \leq C \delta^{-\eta} s^{\kappa/2}\mathcal{E}_{\delta}(A).
\end{equation}
\end{lem}
\begin{proof}
For each index $i$, let $\tilde A_i$ be the $2\delta$-neighborhood of $A_i$ and let $\tilde A = \tilde A_1\times\ldots\times\tilde A_d$. Then for each $t\in [0,1]$, 
\begin{equation}\label{compareCoveringNumberVolume}
\mathcal{E}_{\delta}\big(  A\cap N_{s}(\{\phi = t\}) \big)\lesssim \delta^{-d}|\tilde A \cap N_s(\{\phi=t\})|.
\end{equation}

Apply Corollary \ref{coveringNbhAnalyticSet} to the analytic function $\Phi(z,t) = \phi(z)-t$ (which has domain $U\times\RR\supset [0,1]^{d+1}$) and the sets $\tilde A_1,\ldots,\tilde A_d,$ and $\tilde A_{d+1}=[0,1]$ (note that $\tilde A_{d+1}$ satisfies the non-concentration hypothesis of Corollary \ref{coveringNbhAnalyticSet}). We have that
\begin{equation}\label{consequenceOfCorcoveringNbhAnalyticSet}
\mathcal{E}_{\delta}\big( (\tilde A \times \tilde A_{d+1})\cap N_s(Z(\Phi)) \big) \leq C_0 \delta^{-\eta} s^{\kappa}\mathcal{E}_{\delta}(\tilde A \times \tilde A_{d+1}).
\end{equation}
Since each $\tilde A_i$ is a union of intervals of length $2\delta$ and $s>\delta$, we have
\begin{equation*}
\mathcal{E}_{\delta}\big( (\tilde A \times \tilde A_{d+1})\cap N_s(Z(\Phi)) \big)\sim\delta^{-(d+1)}\vol_{d+1}\big( (\tilde A \times \tilde A_{d+1})\cap N_s(Z(\Phi)) \big),
\end{equation*}
and 
\begin{equation*}
\mathcal{E}_{\delta}(\tilde A \times \tilde A_{d+1})\sim\delta^{-(d+1)}\vol_{d+1}(\tilde A \times \tilde A_{d+1}).
\end{equation*}
Thus by \eqref{consequenceOfCorcoveringNbhAnalyticSet}, we have
\begin{equation}
\begin{split}
\int_0^1 \vol_{d}\big( \tilde A\cap N_s(\{\phi = t\} ) \big)dt
&=\vol_{d+1}\big(\bigcup_{t\in [0,1]} (\tilde A\cap N_s(\{\phi = t\}) \big)\\
&\leq \vol_{d+1}\big( (\tilde A \times\tilde A_{d+1})\cap N_s(Z(\Phi))\big)\\
& \lesssim C_0 \delta^{-\eta} s^{\kappa}\vol_{d+1}(\tilde A \times \tilde A_{d+1})\\
& = C_0 \delta^{-\eta} s^{\kappa}\vol_d(\tilde A ),
\end{split}
\end{equation}
where on the last line we used the fact that $\tilde A_{d+1}=[0,1]$. Thus there is a constant $C_1$ so that for each $a\geq 1$, 
\begin{equation*}
\Big|\Big\{ t\in [0,1]\colon \vol_{d}\big( \tilde A \cap N_s(\{\phi = t\} ) \big)\geq aC_1 \delta^{-\eta} s^{\kappa}\vol_d(\tilde A ) \Big\}\Big| \leq 1/a.
\end{equation*}
Selecting $a = s^{-\kappa/2}$, and using \eqref{compareCoveringNumberVolume}, we conclude that there exists $t\in [t_0,2t_0]$ so that
\begin{equation*}
\begin{split}
\mathcal{E}_{\delta}\big(  A\cap N_{s}(\{\phi = t\}) \big) 
& \lesssim \delta^{-d}|\tilde A \cap N_s(\{\phi=t\})|\\
& \lesssim C_1 \delta^{d-\eta} s^{\kappa/2}|\tilde A|\\
& \lesssim C_1 \delta^{-\eta} s^{\kappa/2}\mathcal{E}_{\delta}(A).\qedhere
\end{split}
\end{equation*}
\end{proof}

\subsection{Cutting non-concentrated sets into cubes}
In this section, we will show that if $\phi\colon U \to\RR$ is real analytic and if $A=A_1\times\ldots\times A_d$ is a product of non-concentrated sets contained in $U$, then most of $A$ can be covered by a small number of axis-parallel cubes, so that $\phi$ is well behaved on each of these cubes.

\begin{lem}\label{semiAnalyticBoundaryCubeDecomp}
Let $U\subset\RR^d$ be an open set that contains $[0,1]^d$ and let $\phi\colon U \to\RR$ be real analytic and not identically zero. Let $V\subset [0,1]^d$ and suppose that $\operatorname{bdry}(V) \subset Z(\phi)$. 
Let $A_1,\ldots,A_d\subset[0,1]$ be sets and let $A = A_1\times\cdots\times A_d$. Let $0<\eta<\kappa$, let $\delta>0$, and suppose that for all intervals $J$ of length at least $\delta$, each set $A_i$ satisfies the non-concentration condition
\begin{equation*}
\mathcal{E}_{\delta}(A_i \cap J) \leq \delta^{-\eta}|J|^{\kappa}\mathcal{E}_{\delta}(A_i).
\end{equation*}
%
Then for each
\begin{equation}\label{restrictionOnP}
p\geq  \delta^{\kappa-\eta}\frac{\mathcal{E}_{\delta}(A)}{\mathcal{E}_{\delta}(A\cap V)},
\end{equation}
there exists a set $\mathcal{Q}$ of cubes with disjoint interiors, each of which is contained in $S$, with 
\begin{equation}\label{notTooManyCubes}
\#\mathcal{Q}\leq  \delta^{-\frac{d \eta}{\kappa}}p^{\frac{-d}{\kappa}}\Big(\frac{\mathcal{E}_{\delta}(A\cap V)}{\mathcal{E}_{\delta}(A)}\Big)^{-\frac{d}{\kappa}},
\end{equation}
and 
\begin{equation}\label{SsetminusCubesSmall}
\mathcal{E}_{\delta}\left(A\cap V\setminus\bigcup_{Q\in\mathcal{Q}}Q\right)\lesssim p\mathcal{E}_{\delta}(A\cap V). 
\end{equation} 
\end{lem}
\begin{proof}
By Corollary \ref{coveringNbhAnalyticSet}, for each $s>\delta$ we have 
\begin{equation}\label{AcapNcor}
\mathcal{E}_{\delta}(A\cap N_s(Z(\phi)))\lesssim  \delta^{-\eta} s^{\kappa}\mathcal{E}_\delta(A).
\end{equation}
Using \eqref{restrictionOnP}, we can select $s$ satisfying
\begin{equation}
\label{valueOfTSubCubeProp}
\delta < s=  
\delta^{\frac{\eta}{\kappa}}p^{\frac{1}{\kappa}} \Big(\frac{\mathcal{E}_{\delta}(A\cap V)}{\mathcal{E}_{\delta}(A)}\Big)^{\frac{1}{\kappa}}.
\end{equation}
The inequality \eqref{AcapNcor} then gives 
\begin{equation}\label{AcapUtSmall}
\mathcal{E}_{\delta}(A\cap N_s(Z(\phi))) \lesssim p\mathcal{E}_\delta(A \cap V).
\end{equation} 


Note that if $p$ is sufficiently small, \eqref{AcapUtSmall} implies that $V$ has non-empty interior. Let $V^{\circ}$ be the interior of $V$. If $V^{\circ}$ is empty, there is nothing to prove, and the lemma follows. Otherwise, apply Theorem \ref{whitneyCubeDecomp} to $V^{\circ}$, and let $\mathcal{Q}_0$ be the resulting collection of Whitney cubes. Let $\mathcal{Q}\subset\mathcal{Q}_0$  be the set of cubes that are not contained in $N_s(Z(\phi))$, for the value of $s$ given by \eqref{valueOfTSubCubeProp}.
Since 
\begin{equation*}
V\setminus\bigcup_{Q\in\mathcal{Q}}Q\subset N_s(Z(\phi)),
\end{equation*}
\eqref{AcapUtSmall} implies \eqref{SsetminusCubesSmall}. Since no cube from $\mathcal{Q}$ is contained in $N_s(Z(\phi))$, each cube in $\mathcal{Q}$ has sidelength $\gtrsim s$. Since the cubes are interior-disjoint and contained in $V\subset [0,1]^d$, we have $\#\mathcal{Q}\lesssim s^{-d}$, which implies \eqref{notTooManyCubes}.
\end{proof}


\begin{lem}\label{cubesWhereAnalyticFLarge}
Let $U\subset\RR^d$ be an open set that contains $[0,1]^d$ and let $f_1,\ldots,f_k\colon U\to\RR$ be real analytic functions, none of which are identically zero. Then there is a constant $C$ so that the following holds. 

Let $A_1,\ldots,A_d\subset[0,1]$ be sets and let $A = A_1\times\cdots\times A_d$. Let $\delta>0$ and suppose that for all intervals $J$ of length at least $\delta$, each set $A_i$ satisfies the non-concentration condition
\begin{equation*}
\mathcal{E}_{\delta}(A_i \cap J) \leq \delta^{-\eta}|J|^{\kappa}\mathcal{E}_{\delta}(A_i).
\end{equation*}

Let $0<w<\kappa/2$. Then there is a set of cubes $\mathcal{Q}$ with disjoint interiors that are contained in $[0,1]^d$, with $\#\mathcal{Q} \lesssim \delta^{-2dw/\kappa}$, so that for each cube $Q\in\mathcal{Q}$ and each index $j$, there is a number $v_{Q,j}\geq  \delta^{w}$ so that
\begin{equation}\label{fjRoughlyConstOnQ}
v_{Q,j}\leq f_j(z)< 4v_{Q,j}\quad\textrm{for all}\ z\in Q.
\end{equation}

Furthermore,
\begin{equation}\label{cubeHasLargeVolume}
\mathcal{E}_{\delta}\big(A\ \backslash\  \bigcup_{Q\in\mathcal{Q}}Q\big) \leq C \delta^{\frac{w\kappa}{C}-\eta}\mathcal{E}_{\delta}(A).
\end{equation}
In the above, the implicit constants in the quasi-inequalities may depend on $d,\kappa,s$, and $f_1,\ldots,f_k$; but they are independent of $\delta$.
 
Finally, if the functions $f_1,\ldots,f_k$ are polynomials, then we can take $C$ to depend only on $d$ and $\max(\deg f_1,\ldots,\deg f_k)$. 
\end{lem}
\begin{proof}
By Theorem \ref{Lojasiewicz}, there is a constant $C\ge 1$ so that for each index $1\leq j\leq k$ and each $z\in [0,1]^d$, we have
\begin{equation*}
|f_j(z)| \gtrsim {\rm dist}(z, Z(f_j) )^{C}.
\end{equation*}
If $f_1,\ldots,f_k$ are polynomials, then we can take $C$ to depend only on $d$ and $\max(\deg f_1,\ldots,\deg f_k)$.

Since $f_j$ is not identically zero, by 
Corollary~\ref{coveringNbhAnalyticSet} we have that for each index $j=1,\ldots,k$ 
\begin{equation}\label{coveringNeighborhoodLevelSet}
\mathcal{E}_{\delta}\big(A \cap N_s( Z(f_j))\big)\lesssim \delta^{-\eta}s^{\kappa}\mathcal{E}_{\delta}(A).
\end{equation}

Thus we can select a choice of $s$ with
\begin{equation}\label{valueOfT}
s\sim  \delta^{\frac{w}{C}}
\end{equation}
so that for each index $j$ we have $|f_j| > \delta^w$ on $[0,1]^d\backslash N_s( Z(f_j))$, and
\begin{equation}\label{WjtSmallMeausre}
\mathcal{E}_{\delta}(A \cap N_{2s}( Z(f_j) )\lesssim \delta^{\frac{w\kappa}{C}-\eta}\mathcal{E}_{\delta}(A).
\end{equation}

Let $L\sim|\log\delta|$ be chosen sufficiently large so that $\sup_{z\in [0,1]^d}|f_j(z)|\leq 2^L \delta^w$ for each index $j$. Define
\begin{equation}\label{valueOfSPrime}
s'= \delta^{2w/\kappa}.
\end{equation}
For each index $j$ and each integer $1\leq \ell\leq L$, use Lemma \ref{uniformBdPhiT} to select a choice of $t_{j,\ell}\in [2^\ell\delta^w, 2^{\ell+1}\delta^w]$ so that
\begin{equation}\label{NeighborhoodLevelSetstjl}
\mathcal{E}_{\delta}\big(  A\cap N_{2s'}(\{f_j = t_{j,\ell}\}) \big) \leq C_0 \delta^{-\eta} (s')^{\kappa/2}\mathcal{E}_{\delta}(A).
\end{equation}
Note that for each $\ell\geq 0$, $t_0 = 2^\ell\delta^w$ satisfies $(s')^{\kappa/2}<t_0$, so the hypotheses of Lemma \ref{uniformBdPhiT} are satisfied.

For $i=1,2$, define
\begin{equation*}
K_i =  \bigcup_{j=1}^k N_{is}( Z(f_j))\ \cup \ \bigcup_{j=1}^k \bigcup_{\ell=1}^L N_{i s'}(\{f_j = t_{j,\ell}\}). 
\end{equation*}
Then combining \eqref{WjtSmallMeausre}  and \eqref{NeighborhoodLevelSetstjl} 
we have 
\begin{equation}\label{WjtSmallMeausre2}
\begin{split}
\mathcal{E}_{\delta}(A \cap K_2) & \leq \sum_{j=1}^k\mathcal{E}_{\delta}\big(A \cap N_{2s}( Z(f_j) )\big)\ +\ \sum_{j=1}^k\sum_{\ell=1}^L  \mathcal{E}_{\delta}\big(A \cap N_{2s'}( Z(\{f_j=t_{j,\ell}\}))\big)\\
&\lesssim   \delta^{\frac{w\kappa}{C}-\eta}\mathcal{E}_{\delta}(A) + |\log\delta|  \delta^{-\eta}(\delta^{2w/\kappa})^{\kappa/2}\mathcal{E}_{\delta}(A)\\
&\lesssim \delta^{\frac{w\kappa}{C}-\eta}\mathcal{E}_{\delta}(A).
\end{split}
\end{equation} 

Let $\mathcal{Q}_0$ be the set of cubes obtained by applying Theorem \ref{whitneyCubeDecomp} to the interior of $[0,1]^d\backslash \overline K_1$, and let $\mathcal{Q}\subset\mathcal{Q}_0$ be the set of all cubes that intersect $[0,1]^d\backslash K_2$; we have that each cube $Q\in\mathcal{Q}$ has side-length $\gtrsim s'$, so in particular $\#\mathcal{Q}\lesssim (s')^{-d}\lesssim \delta^{-2dw/\kappa}$.

By construction, for each index $j$ and each connected component $U\subset [0,1]^d\backslash \overline K_1$ we have
\begin{equation*}
\sup_{z\in U} f_j(z) \leq 4 \inf_{z\in U}f_j(z),
\end{equation*}
and thus for each cube $Q\in\mathcal{Q}$ there is a number $v_{Q,j}$ so that \eqref{fjRoughlyConstOnQ} holds. Since $|f_j|\geq\delta^w$ on $[0,1]^d\backslash \overline K_1$, we can take $v_{Q,j}\geq\delta^w$. 

Finally, \eqref{cubeHasLargeVolume} follows from \eqref{WjtSmallMeausre2} and the observation that
\begin{equation*}
[0,1]^d\ \backslash \bigcup_{Q\in\mathcal{Q}}Q \subset K_2. \qedhere
\end{equation*}
\end{proof}

\section{Shmerkin's nonlinear discretized projection theorem and its consequences}\label{ShmerkinSection}
In this section we will discuss Shmerkin's nonlinear generalization \cite{Shm} of Bourgain's discretized projection theorem \cite{B10}. Shmerkin's nonlinear projection theorem concerns smooth functions of the form $G\colon Q\times I\to\RR$, where $Q\subset[0,1]^2$ is a square and $I\subset[0,1]$ be an interval. We will call these ``projection functions.''

In what follows, it will be helpful to introduce some additional notation. If $G\colon Q\times I\to \RR$ is a projection function, define $G_{(z)}(x,y) = G(x,y,z)$. In particular, 
\begin{equation*}
\nabla G_{(z)}(x,y)=\big( \partial_x  G(x,y,z),\ \partial_y  G(x,y,z)\big),
\end{equation*} 
so $\nabla G_{(z)}(x,y)$ is a vector in $\RR^2$. For such a function $G$, and for $(x,y)\in Q$ and $z\in I$, define the map
\begin{equation}\label{defnTheta}
\theta_{(x,y)}(z) = \angle\operatorname{dir}(\nabla G_{(z)}(x,y)).
\end{equation}

With these definitions, we can now state Shmerkin's result from \cite{Shm}.
\begin{thm}[Shmerkin]\label{Shm}
For every $\eta>0$ and $C>0$, there exists $\tau=\tau(\eta)>0$ such that the following holds for all $\delta>0$ sufficiently small compared to $\eta$ and $C$. Let $Q\subset[0,1]^2$ be a square and let $X\subset Q$ be a union of $\delta$-squares. Let $I\subset[0,1]$ be an interval, and let $Z\subset I$ be a union of $\delta$-intervals.

Let $G\colon Q\times I\to\RR$ be a projection function that satisfies 
\begin{equation}\label{boundsOnG}
\sup_{z\in Z}\Vert G_{(z)}\Vert_{C^2(Q)}<C,\quad\quad\inf_{(x,y,z)\in X\times Z}|\nabla G_{(z)}(x,y)|>C^{-1}.
\end{equation}

Suppose that $X$ satisfies the non-concentration estimate
\begin{equation}\label{Xnonconcentration}
|X \cap Q^\prime|\leq \delta^{\eta}|X|,
\end{equation}
whenever $Q^\prime$ is a square of side-length $|X|^{1/2}$. Suppose furthermore that for each $(x,y)\in X$ and each direction $v\in S^1$, we have the ``transversality'' estimate
\begin{equation}\label{thetaimagenoncon}
|\{z\in Z\colon |\theta_{(x,y)}(z)-v|\leq r\}| \leq \delta^{-\tau}r^{\eta}|Z|\quad\textrm{for all}\ \delta\leq r\leq 1.
\end{equation}

Then there is a small bad set $Z_{\operatorname{bad}}\subset Z$ with $|Z_{\operatorname{bad}}|\leq \delta^{\tau}|Z|$ so that for all $z\in Z\backslash Z_{\operatorname{bad}}$, the following holds: Let $X^\prime\subset X$ with $|X^\prime|\geq \delta^{\tau}|X|$. Then
\begin{equation}
\mathcal{E}_{\delta}(\{G(x,y,z)\colon (x,y)\in X^\prime\})\geq\delta^{-\tau}\mathcal{E}_{\delta}(X)^{1/2}.
\end{equation}
\end{thm}

Our task for the remainder of this section is to reformulate Theorem \ref{Shm} as a statement about energy dispersion.

\subsection{Weakening the hypotheses of Theorem~\ref{Shm}}
Our first task is to state a version of Theorem~\ref{Shm} where the bounds \eqref{boundsOnG} are slightly less restrictive. We will also recast several of the hypotheses and conclusions to align more closely with the setup of Theorem \ref{main-energy-dispersion}. 
\begin{lem}\label{Shmcor}
For every $0<\alpha<1$ and $0<\kappa\leq\alpha$, there exists $\sigma=\sigma(\alpha,\kappa)>0$ such that the following holds. 

 Let $I\subset[0,1]$ be an interval, let $Q\subset[0,1]^2$ be a square, let $\delta>0$, and let $G\colon Q\times I \to\RR$ be a projection function that satisfies the non-degeneracy conditions
\begin{align}
\sup_{z\in I}\Vert G_{(z)}\Vert_{C^2(Q)}&\leq\delta^{-\sigma},\label{boundOnC2NormG}\\
\sup_{(x,y,z)\in Q\times I}|\partial_z G(x,y,z)|&\leq \delta^{-\sigma},\label{boundOnGz}\\
\inf_{(x,y,z)\in Q\times I}|\nabla G_{(z)}(x,y)|& \geq \delta^\sigma,\label{boundOnNablaG} \\
\inf_{(x,y,z)\in Q\times I} |\partial_z \theta_{(x,y)}(z)|&\geq \delta^{\sigma},  \label{lowerBoundThetap}
\end{align}
and suppose that $\nabla G_{(z)}(x,y)$ is roughly constant on $Q\times I$, i.e.
\begin{equation}\label{nablaGRoughlyConst}
\sup_{(x,y,z)\in Q\times I} |\nabla G_{(z)}(x,y)| \leq 4 \inf_{(x,y,z)\in Q\times I}|\nabla G_{(z)}(x,y)|.
\end{equation}

Let $A_1,A_2,A_3$ be unions intervals of length $\delta$ with $A_1\times A_2\subset Q$ and $A_3\subset I$. Suppose that for each index $i=1,2,3$  we have $|A_i|\le \delta^{1-\alpha}$, and that for each interval $J$ we have the non-concentration condition
\begin{equation}\label{nonconcorA1234}
|A_i\cap J|\le \delta^{-\sigma}|J|^\kappa|A_i|.
\end{equation}
Let $R\subset\RR$ be a set with $\mathcal{E}_{\delta}(R)\leq \delta^{-\alpha-\sigma}$.
Then 
\begin{equation}\label{volOfTriples}
\mathcal{E}_{\delta}\big( \{(a_1,a_2,a_3)\in A_1\times A_2\times A_3\colon G(a_1,a_2,a_3)\in R\} \big) \lesssim \delta^{-3\alpha+\sigma}.
\end{equation}
\end{lem}
\begin{proof}
Our proof consists of three main steps. While there are a few technical details, each of these steps follow standard arguments. First, note that our projection function $G$ satisfies the conditions \eqref{boundOnC2NormG} and \eqref{boundOnNablaG}. These are different from the requirement \eqref{boundsOnG} that is needed when applying Theorem \ref{Shm}.  We will apply a re-scaling argument to create a new function $\tilde G$ that satisfies \eqref{boundsOnG}. 

Second, our projection function $G$ satisfies the condition \eqref{lowerBoundThetap}, and our sets $A_1,A_2,A_3$ satisfy the non-concentration condition \eqref{nonconcorA1234}. These are different from the corresponding requirements \eqref{thetaimagenoncon} and \eqref{Xnonconcentration}. We will verify that our new function $\tilde G$ and appropriately constructed sets $X$ and $Z$ satisfy \eqref{Xnonconcentration} and \eqref{thetaimagenoncon}.

Third, we will show that our desired bound \eqref{volOfTriples} follows by applying Theorem \ref{Shm} to our newly constructed function $\tilde G$. In brief, if the bound \eqref{volOfTriples} failed then there must exist a large subset $A_3^\prime\subset A_3$ so that $G(A_1,A_2,z)$ has small $\delta$-covering number for each $z\in A_3^\prime$. We will show that this contradicts Theorem \ref{Shm}. 

\medskip

\noindent \textbf {Constructing $\tilde G$: A re-scaling argument}.
First we will find a slightly smaller sub-square of $Q$ and sub-interval of $I$ that still captures most of the triples from \eqref{volOfTriples}. Let $Q^\prime\subset Q$ be a square of side-length at most $\delta^{2\sigma}$ and $I'\subset I$ an interval of length at most $\delta^{4\sigma}$ that maximizes the quantity
\begin{equation*}
\mathcal{E}_{\delta}\big( \{(a_1,a_2,a_3)\in (Q^\prime\cap (A_1\times A_2))\times (I^\prime\cap A_3)\colon G(a_1,a_2,a_3)\in R\} \big).
\end{equation*}
Define $A_1^\prime\times A_2^\prime =Q^\prime\cap  (A_1\times A_2)$ and $A_3^\prime=I'\cap A_3$.
Since $Q\times I\subset [0,1]^3$, by pigeonholing, we have
\begin{equation}\label{mostEntropyCapturedSubSquare}
\begin{split}
\mathcal{E}_{\delta}&\big( \{(a_1,a_2,a_3)\in A_1\times A_2\times A_3\colon G(a_1,a_2,a_3)\in R\} \big) \\
&\leq \delta^{-8\sigma} \mathcal{E}_{\delta}\big( \{(a_1,a_2,a_3)\in A_1^\prime \times A_2^\prime\times A_3^\prime\colon G(a_1,a_2,a_3)\in R\} \big). 
\end{split}
\end{equation}

Let
\begin{equation*}
m = \inf_{(x,y,z)\in Q^\prime \times I^\prime}|\nabla G_{(z)}(x,y)|.
\end{equation*}
By \eqref{boundOnNablaG} we have $m\geq\delta^{\sigma}$.

Let $(x_0,y_0)$ be the bottom-left corner of $Q^\prime$, and let $\tilde Q  = \delta^{-2\sigma}( Q^\prime  -(x_0,y_0))$; thus $\tilde Q$ is a sub-square contained in $[0,1]^2$ (possibly $\tilde Q = [0,1]^2$) with bottom left endpoint $(0,0)$. Let $z_0$ be the leftmost point in $I^\prime$ and define $\tilde I = I^\prime - z_0$. 
For $(x,y)\in \tilde Q$ and $z\in \tilde I$, define 
\begin{equation*}
\tilde G(x,y,z) = m^{-1}\delta^{-2\sigma}( G(\delta^{2\sigma} x+x_0, \delta^{2\sigma} y+y_0, z+z_0)-G(x_0,y_0,z_0)).
\end{equation*}
We have $\tilde G\colon \tilde Q\times \tilde I\to\RR$, with $\tilde G(0,0,0)=0$. Indeed, if $(x,y)\in\tilde Q$ and $z\in \tilde I$, then $(\delta^{2\sigma} x+x_0, \delta^{2\sigma} y+y_0)\in Q^\prime$, $z+z_0\in I^\prime$, and thus $G(\delta^{2\sigma} x+x_0, \delta^{2\sigma} y+y_0, z+z_0)-G(x_0,y_0,z_0)$ is well-defined. 

We have 
\begin{equation}\label{rescaledGTilde}
\begin{split}
\partial_x \tilde G(x,y,z) & = m^{-1} G_1(\delta^{2\sigma} x + x_0, \delta^{2\sigma} y + y_0, z+z_0),\\
\partial_y \tilde G(x,y,z) & = m^{-1} G_2(\delta^{2\sigma} x + x_0, \delta^{2\sigma} y + y_0, z+z_0),
\end{split}
\end{equation}
where for clarity we write $G_1(\cdot,\cdot,\cdot)$ (resp.~$G_2(\cdot,\cdot,\cdot)$ ) to denote the partial derivative of $G$ with respect to its first (resp.~second) variable; this notation will not be used anywhere else in the paper. \eqref{rescaledGTilde} implies that 
\begin{equation}\label{infGradientTildeG}
\inf_{(x,y,z)\in \tilde Q \times \tilde I}|\nabla \tilde G_{(z)}(x,y)|\geq 1.
\end{equation}
On the other hand, by \eqref{nablaGRoughlyConst} we have
\begin{equation*}
\sup_{(x,y,z)\in \tilde Q \times \tilde I}|\nabla \tilde G_{(z)}(x,y)|\leq 4.
\end{equation*}
This implies 
$
|\tilde G(x, y, z)-\tilde G(0,0,z)| \le  4 \sqrt2
$ for all $(x, y, z) \in \tilde Q\times \tilde I$.
By \eqref{boundOnGz}, we have
$|\tilde G(0,0,z)|
\le \delta^{-4\sigma}|z|\le 1,$
for all $z\in \tilde I$.
Thus
\begin{equation}
|\tilde G(x, y, z)| \le  4 \sqrt2+1\quad\textrm{for all }\ (x, y, z) \in \tilde Q\times \tilde I.
\end{equation}

Finally, we can compute
\begin{equation*}
|\partial_{xx}\tilde G(x,y,z)| = |m^{-1}\delta^{2\sigma} (\partial_{xx} G)(\delta^{2\sigma} x + x_0, \delta^{2\sigma} y + y_0, z+z_0)|\leq 1,
\end{equation*}
where the final inequality used the bound $m\geq\delta^\sigma$ and \eqref{boundOnC2NormG}. A similar computation shows that the other second derivatives of $\tilde G$ with respect to $x$ and $y$ are bounded by 1, and thus
\begin{equation}\label{C2NormTildeG}
\sup_{z\in \tilde I }\Vert \tilde G_{(z)}\Vert_{C^2(\tilde Q)}\leq 100.
\end{equation}
Finally, if $(x,y)\in\tilde Q$ and $z\in \tilde I$, then
\begin{equation*}
\tilde\theta_{(x,y)}(z)=\theta_{(\delta^{2\sigma} x + x_0, \delta^{2\sigma} y + y_0)}(z+z_0),
\end{equation*}
so for each $(x,y,z)\in\tilde Q\times \tilde I$, \eqref{lowerBoundThetap} implies
\begin{equation}\label{sizeThetaTildeG}
|\partial_z \tilde \theta_{(x,y)}(z)|\geq \delta^{\sigma}.
\end{equation}

Let $\tilde A_1 = \delta^{-2\sigma}(A_1^\prime-x_0)$, $\tilde A_2 = \delta^{-2\sigma}(A_2^\prime-y_0)$, and $\tilde A_3 = A_3^\prime-z_0$. Observe that the sets $\tilde A_1,\tilde A_2,\tilde A_3$ are still unions of intervals of length $\delta$, so in particular $\mathcal{E}_{\delta}(\tilde A_i)\sim\delta^{-1}|\tilde A_i|$. We also have that 
\begin{equation}\label{volBoundAi}
|\tilde A_i|\leq \delta^{1-\alpha-2\sigma}.
\end{equation}
The sets $\tilde A_1$ and $\tilde A_2$ satisfy a slightly weaker version of \eqref{nonconcorA1234}---if $J\subset[0,1]$ is an interval , then
\begin{equation}\label{nonConcentrationTildeAi}
\begin{split}
|\tilde A_i\cap J| &= \delta^{-2\sigma}|A_i \cap \delta^{2\sigma}(J+x_0)|\\
&\leq \delta^{-2\sigma} (\delta^{-\sigma} (\delta^{2\sigma}|J|)^{\kappa}|A_i|)\\
&\leq \delta^{-(2+\kappa)\sigma}|J|^{\kappa}\delta^{1-\alpha}.
\end{split}
\end{equation}
The set $\tilde A_3$ is just a translate of $A_3^\prime$, so, in view of \eqref{nonconcorA1234}, we have
\begin{equation*}
|\tilde A_3\cap J| \le \delta^{-\sigma}|J|^\kappa\delta^{1-\alpha}.
\end{equation*}

Finally, let $\tilde R = m^{-1}\delta^{-2\sigma}(R - G(x_0,y_0,z_0))$ and note that $\mathcal{E}_{\delta}(\tilde R)\le \delta^{-\alpha-4\sigma}$.  Then 
\begin{equation}\label{entropyCapturedTildeAi}
\begin{split}
\mathcal{E}_{\delta}&\big( \{(a_1,a_2,a_3)\in A_1^\prime \times A_2^\prime\times A_3^\prime\colon G(a_1,a_2,a_3)\in R\} \big)\\
&\leq\delta^{4\sigma} \mathcal{E}_{\delta}\big( \{(a_1,a_2,a_3)\in \tilde A_1 \times \tilde A_2 \times \tilde A_3 \colon \tilde G(a_1,a_2,a_3)\in \tilde R\} \big).
\end{split}
\end{equation}
Combining \eqref{mostEntropyCapturedSubSquare} and \eqref{entropyCapturedTildeAi}, we see that in order to establish \eqref{volOfTriples}, it suffices to prove the bound
\begin{equation}\label{modifiedVolOfTriples}
\mathcal{E}_{\delta}\big( \{(a_1,a_2,a_3)\in \tilde A_1 \times \tilde A_2 \times \tilde A_3 \colon \tilde G(a_1,a_2,a_3)\in \tilde R\} \big)\lesssim \delta^{-3\alpha+5\sigma}.
\end{equation}
We will establish the bound \eqref{modifiedVolOfTriples} by applying Theorem \ref{Shm} to $\tilde G$. To do so, we must verify that $\tilde G$, $\tilde A_1$, $\tilde A_2$, and $\tilde A_3$ satisfy the hypotheses of Theorem \ref{Shm}.

\medskip
\noindent \textbf{Verifying the conditions for Theorem \ref{Shm}}
Let $\eta = (1-\alpha)\kappa$, and let $\tau>0$ be the value obtained by applying Theorem \ref{Shm} with this choice of $\eta$. Let 
\begin{equation}\label{sigmasmall}
\sigma < \min\big( (1-\alpha)\kappa/20,\ \tau/9\big).
\end{equation}
Define
\begin{equation*}
T:=\{(a_1,a_2,a_3)\in \tilde A_1\times \tilde A_2\times \tilde A_3\colon \tilde G(a_1,a_2,a_3)\in \tilde R\}.
\end{equation*}
Since $\tilde A_1,\tilde A_2$ an $\tilde A_3$ are unions of $\delta$ intervals, to obtain \eqref{modifiedVolOfTriples} it suffices to prove that
\begin{equation}\label{desiredBoundOnT}
\vol_3(T)\leq \delta^{3-3\alpha+5\sigma}.
\end{equation}

First, we can suppose that for each index $i$, 
\begin{equation}\label{AiNotTooSmall}
\vol_1(\tilde A_i)\geq \delta^{1-\alpha+5\sigma}.
\end{equation} 
Indeed, if \eqref{AiNotTooSmall} fails for some index $i$ then \eqref{desiredBoundOnT} follows from the trivial bound 
\begin{equation*}
\begin{split}
\vol_3(T) &\leq \vol_3(\tilde A_1 \times \tilde A_2 \times \tilde A_3)\\
&\leq \delta^{2-2\alpha-4\sigma}\min_{1\leq i\leq 3}\vol_1(A_i),
\end{split}
\end{equation*}
where on the last line we used \eqref{volBoundAi}.

Let $X=\tilde A_1\times \tilde A_2$ and let $Z = \tilde A_3$. We will verify that $\tilde G$ and the sets $X$ and $Z$ satisfy the hypotheses of Theorem \ref{Shm} with parameter $\eta$. To begin, \eqref{boundsOnG} follows from \eqref{infGradientTildeG} and \eqref{C2NormTildeG}. 

Next we will verify that $X$ satisfies the non-concentration estimate \eqref{Xnonconcentration}. Let $r = \vol_2(X)^{1/2}\leq \delta^{1-\alpha-2\sigma}$. Let $B(x,r)$ be a ball of radius $r$, and let $J_1\times J_2$ be a square of side-length $2r$ that contains $B(x,r)$. Then by \eqref{nonConcentrationTildeAi}, we have
\begin{equation*}
\begin{split}
\vol_2(X\cap B(x,r) )
&\le \vol_1(\tilde A_1\cap J_1)\vol_1(\tilde A_2\cap J_2)\\
&\leq \big( (\delta^{-3\sigma}(2\delta^{1-\alpha-2\sigma})^{\kappa}\delta^{1-\alpha}\big)^2\\
&\leq 4^{\kappa}\frac{\delta^{2-2\alpha-6\sigma-4\kappa \sigma}}{|\tilde A_1| |\tilde A_2|}\delta^{2\eta}|\tilde A_1| |\tilde A_2|\\
&\leq \delta^{\eta} \vol_2(X),
\end{split}
\end{equation*}
where for the last line we used \eqref{sigmasmall} and \eqref{AiNotTooSmall}.

Finally, we will verify the transversality estimate \eqref{thetaimagenoncon}. Note that for each $(x,y)\in \tilde Q$, the function $\tilde\theta_{(x,y)}(z)$ is monotone and continuous on $\tilde I$. Thus for each interval $K$, the pre-image
\begin{equation*}J:=\theta_{(x,y)}^{-1}(K)\end{equation*} 
is  necessarily an interval.
Moreover, \eqref{sizeThetaTildeG} implies that for every $z_1,z_2\in J$, we have
\begin{equation*}
\frac{|\tilde \theta_{(x,y)}(z_1)-\tilde \theta_{(x,y)}(z_2)|}{|z_1-z_2|}\ge \delta^{\sigma}.
\end{equation*}
Thus $|J|\le \delta^{-\sigma}| K|$. 
It follows that for every interval $K\subset \RR$ of length at most $2r$, 
we have
\begin{align*}
|\{z\in \tilde A_3\colon \tilde \theta_{(x,y)}(z)\in K\} |
&=|\tilde A_3\cap J| \\
&\leq \delta^{-\sigma}|J|^{\kappa}\delta^{1-\alpha}\\
&\leq \delta^{-\sigma-\kappa \sigma}|K|^{\kappa}\frac{\delta^{1-\alpha}}{|\tilde A_3|}|\tilde A_3|\\ 
&\le \delta^{-(6+\kappa)\sigma}|\tilde A_3|r^{\eta}r^{\alpha\kappa}2^{\kappa}\\
&\le \delta^{-\tau}|\tilde A_3|r^{\eta},
\end{align*}
where we used \eqref{sigmasmall} and \eqref{AiNotTooSmall}.
Thus \eqref{thetaimagenoncon} holds.

\medskip

\noindent \textbf{Applying Theorem \ref{Shm}}
Apply Theorem \ref{Shm} to $\tilde G$, $X=\tilde A_1\times \tilde A_2$, and $Z=\tilde A_3$. We obtain a set $(\tilde A_3)_{\operatorname{bad}}\subset  \tilde A_3$ with
\begin{equation*}
\vol_1((\tilde A_3)_{\operatorname{bad}})\leq \delta^{\tau}\vol_1(\tilde A_3).
\end{equation*}

For each $c\in \tilde A_3$, define
\begin{equation*}
X_c:=\{(x,y)\in \tilde A_1\times \tilde A_2\colon \tilde G(x,y,c)\in \tilde R\}.
\end{equation*}
Let
\begin{equation*}
\tilde A_3':=\left\{z\in \tilde A_3\colon \vol_2(X_z)\ge \tfrac{\vol_3(T)}{2\vol_1(\tilde A_3)}\right\}.
\end{equation*}

Now, suppose \eqref{desiredBoundOnT} fails; we will obtain a contradiction. By Fubini, we have 
\begin{equation*}\vol_3(T)\le \vol_1(\tilde A_3')\vol_1(\tilde A_1)\vol_1(\tilde A_2)+\vol_1(\tilde A_3\setminus \tilde A_3')\cdot \tfrac{\vol_3(T)}{2\vol_1(\tilde A_3)},
\end{equation*}
which implies
\begin{align*}
|\tilde A_3'|
&\ge \frac{\vol_3(T)}{2|\tilde A_1||\tilde A_2|}\\
&=\frac{\vol_3(T)}{2|\tilde A_1||\tilde A_2||\tilde A_3|}|\tilde A_3|\\
&\ge \tfrac12 \delta^{9\sigma}|\tilde A_3|,
\end{align*}
where on the final line we used \eqref{volBoundAi} and the observation $|\tilde A_3|\le |A_3|\le \delta^{1-\alpha}$. 

In particular, since $\sigma>0$ satisfies \eqref{sigmasmall} (and assuming $\delta>0$ is sufficiently small), we have $|\tilde A_3'|> \delta^{\tau}|\tilde A_3|$ and thus
\begin{equation}
\tilde A_3'\not\subseteq (\tilde A_3)_{\rm bad}.
\end{equation}

Fix an element $z\in \tilde A_3'\backslash (\tilde A_3)_{\rm bad}$ and define $X^\prime = X_z$. 
By the definition of $\tilde A_3^\prime$, we have 
\begin{equation}
\vol_2(X_z)\ge
\tfrac{\vol_3(T)}{2|\tilde A_3|}\ge  \tfrac12\delta^{2-2\alpha+5\sigma}\geq \tfrac12 \delta^{9\sigma}\vol_2(X)\geq \delta^{\tau}\vol_2(X),
\end{equation}
and thus by Theorem \ref{Shm} we have
\begin{equation}
\mathcal{E}_{\delta}\big(\{\tilde G(x,y,z)\colon (x,y)\in X_z\}\big) \geq \delta^{-\tau}\mathcal{E}_{\delta}(X)^{1/2}.
\end{equation}
On the other hand, 
\begin{equation}
\mathcal{E}_{\delta}\big(\{\tilde G(x,y,z)\colon(x,y)\in X_z\}\big)
\le \mathcal{E}_{\delta}(\tilde R)
\le  \delta^{-\alpha-4\sigma}
\leq \delta^{-9\sigma}\mathcal{E}_{\delta}(X)^{1/2}\le \delta^{-\tau}\mathcal{E}_{\delta}(X)^{1/2}.
\end{equation}
This is a contradiction. We conclude that \eqref{desiredBoundOnT} holds, which completes the proof.
\end{proof}

\subsection{Reformulating Theorem \ref{Shm} as an energy dispersion estimate}\label{reformulationEnergyDispersionSec}
As the title suggests, in this section we will reformulate Theorem \ref{Shm} as an energy dispersion estimate. The basic idea is that we can count the number of solutions to $P(x,y) = P(x',y')$  with $x,x'\in A,\ y,y'\in B$ by counting the size of the intersection 
\begin{equation}\label{informalIntersection}
\{(x,x',y,y'): P(x,y)-P(x',y')=0\} \cap (A^2\times B^2).
\end{equation}
We will express the surface $Z(P(x,y)-P(x',y'))$ as the graph $y' = G(x,x',y'),$ and then use Lemma \ref{Shmcor} to estimate the size (or more accurately, $\delta$-covering number) of the set \eqref{informalIntersection}. We now turn to the details.

\begin{defn}\label{defnHFDefn}
If $F(x,x',y,y')$ is a function (that is at least twice differentiable on its domain), define
\begin{equation}\label{defnHF}
\begin{split}
H_F(x,x',y,y') &= (\partial_x F)(\partial_{y'} F)(\partial_{x'y}F)-(\partial_xF)(\partial_yF)(\partial_{x'y'}F)\\
&\qquad-(\partial_{x'}F)(\partial_{y'}F)(\partial_{xy}F)+(\partial_{x'}F)(\partial_yF)(\partial_{xy'}F).
\end{split}
\end{equation}
We will be particularly interested in functions $F$ of the form $F(x,x',y,y') = P(x,y) - P(x',y')$, where $P$ is smooth. In this case $H_F$ simplifies to
\begin{equation}\label{defnHFForP}
H_F(x,x',y,y') = \partial_xP(x,y)\partial_yP(x,y)\partial_{x'y'}P(x',y')-\partial_{x'}P(x',y')\partial_{y'}P(x',y')\partial_{xy}P(x,y).
\end{equation}
\end{defn}

\begin{lem}\label{EnergyDispersionForF}
For every $0< \alpha< 1$ and $0 < \kappa \leq \alpha,$ there exists $\eps=\eps(\alpha,\kappa)>0$ such that the following holds. 
Let $F(x,x',y,y')$ be smooth on an open set containing $[0,1]^4$. Then the following is true for all $\delta>0$ sufficiently small (depending on $F$, $\alpha$, and $\kappa$). Let $I_1,\ldots,I_4\subset[0,1]$ be intervals. Suppose that $|H_F|\geq\delta^{\eps}$ on $I_1\times\ldots\times I_4$, and that 
\begin{equation}\label{maxAndMinComparable}
\delta^{\eps}\leq \max|\partial_x F|\leq 2 \min|\partial_x F|,
\end{equation}
(where the $\max$ and $\min$ are taken over $I_1\times I_2\times I_3\times I_4$), and similarly for $\partial_{x'} F,\partial_{y} F$, and $\partial_{y'} F$. 
Let $A_i\subset I_i$, $i=1,2,3,4$ be sets. Suppose that for each index $i=1,\ldots,4$ and each interval $J$ of length at least $\delta$, $A_i$ satisfies the non-concentration condition
\begin{equation}\label{noncon4}
\mathcal{E}_{\delta}(A_i\cap J)\le |J|^\kappa \delta^{-\alpha-\eps}.
\end{equation}

Then we have the energy dispersion estimate
\begin{equation}\label{energyNonConcentrationInLem}
\mathcal{E}_{\delta}\big((A_1\times A_2\times A_3\times A_4) \cap Z(F)  \big) \leq  \delta^{-3\alpha+\eps}.
\end{equation}
\end{lem}
\begin{proof}
First, we can suppose that the interior of $I_1\times I_2\times I_3\times I_4$ contains at least one point of $Z(F)$. Indeed, if $Z(F)\cap (I_1\times I_2\times I_3\times I_4)$ is empty then there is nothing to prove. If $Z(F)$ only intersects the boundary of $(I_1\times I_2\times I_3\times I_4)$, then we can enlarge the intervals $I_1,\ldots,I_4$ very slightly, and the hypotheses of the lemma (perhaps with the associated constants weakened by a factor of 2) will still be satisfied on these enlarged intervals. 

Next, observe that for each $(x,x',y)\in I_1\times I_2\times I_3$, there is at most one $y' \in I_4$ so that $F(x,x',y,y')=0$. Furthermore, by the lower bound \eqref{maxAndMinComparable} for $\partial_{y'} F$, we have that
\begin{equation*}
\begin{split}
\mathcal{E}_{\delta}\big((A_1\times A_2\times A_3\times & A_4) \cap Z(F)  \big) \\
&\leq \delta^{-\eps}\mathcal{E}_{\delta}\{(x,x',y)\in A_1\times A_2\times A_3\colon\ \exists\ y'\in A_4\ \textrm{s.t.}\ F(x,x',y,y')=0\}.
\end{split}
\end{equation*}
In particular, to establish \eqref{energyNonConcentrationInLem} it suffices to prove the estimate
\begin{equation}\label{sufficesToBoundProjection}
\mathcal{E}_{\delta}\Big( \pi\big( (A_1\times A_2\times A_3\times  A_4) \cap Z(F) \big) \Big)\leq \delta^{-3\alpha+2\eps},
\end{equation}
where $\pi\colon\RR^4\to\RR^3$ is the projection to the first three coordinates. 

Let $V=\pi\big(Z(F)\cap (I_1\times I_2\times I_3\times I_4)\big)$. Since the interior of $I_1\times I_2\times I_3\times I_4$ contains at least one point of $Z(F)$ and $\partial_{y'}F\neq 0$ on $I_1\times I_2\times I_3\times I_4$, we have that $V$ has non-empty interior, and  
\begin{equation}\label{formulaForBdryU}
\operatorname{bdry}(V)\subset \operatorname{bdry}(I_1\times I_2\times I_3)\ \cup\ \pi\Big( Z(F)\cap \big(I_1\times I_2\times I_3 \times \operatorname{bdry}(I_4) \big) \Big).
\end{equation}

Let $\eps=\eps(\alpha,\kappa)>0$ be a small quantity that will be specified below. If $\mathcal{E}_{\delta}(V\cap (A_1\times A_2\times A_3))\leq \delta^{-3\alpha + 2\eps}$ then \eqref{sufficesToBoundProjection} holds and we are done.  

Henceforth we shall assume that 
\begin{equation}\label{23ACapULarge}
\mathcal{E}_{\delta}(V\cap (A_1\times A_2\times A_3))\geq \delta^{-3\alpha + 2\eps}\ge \delta^{5\eps}\mathcal{E}_{\delta}(A_1\times A_2\times A_3).
\end{equation} 
Select 
\begin{equation*}
p = \delta^{5\eps},
\end{equation*}
and apply Lemma \ref{semiAnalyticBoundaryCubeDecomp} to the sets $A = A_1\times A_2\times A_3$ and $V$ with this choice of $p$. Observe that if $\eps>0$ is selected sufficiently small compared to $\kappa$, and using \eqref{23ACapULarge}, condition \eqref{restrictionOnP} from the statement of Lemma \ref{semiAnalyticBoundaryCubeDecomp} is satisfied.  

We obtain a set $\mathcal{Q}$ of cubes, with 
\begin{equation}\label{numberOfCubesInDecomp}
\#\mathcal{Q}\leq  \delta^{-\frac{3 \eps}{\kappa}}p^{\frac{-3}{\kappa}}\Big(\frac{\mathcal{E}_{\delta}(A\cap V)}{\mathcal{E}_{\delta}(A)}\Big)^{-\frac{3}{\kappa}}\leq  \delta^{-\frac{C_1 \eps}{\kappa}},
\end{equation}
where $C_1$ is an absolute constant, and 
\begin{equation}\label{coveringNumberOutsideDecomp}
\mathcal{E}_{\delta}\left( A \cap V\setminus\bigcup_{Q\in\mathcal{Q}}Q\right)\lesssim p\mathcal{E}_{\delta}(A\cap V)\leq \delta^{-3\alpha+2\eps}.
\end{equation} 
Thus to establish \eqref{sufficesToBoundProjection} it suffices to prove that for each cube $Q\in\mathcal{Q}$, we have
\begin{equation}\label{neededCoveringBdEachCube}
\mathcal{E}_{\delta}\Big( \pi\big( (A_1^Q\times A_2^Q\times A_3^Q\times  A_4) \cap Z(F) \big) \Big) \leq \delta^{-3\alpha+(\frac{C_1}{\kappa}+2)\eps},
\end{equation}
where $A_1^Q\times A_2^Q\times A_3^Q=(A_1\times A_2\times A_3)\cap Q$. 

Fix a cube $Q\in\mathcal{Q}$, and define $G\colon Q \to\RR$ so that $Z(F)\cap (Q\times I_4)$ is the graph of $G$.  
By the implicit function theorem and \eqref{maxAndMinComparable} (and its analogues for $\partial_{x'} F,\partial_{y} F$, and $\partial_{y'} F$),  we have the bound
\begin{equation*}
C^{-1}\delta^{\eps} \le \max_{p\in Q}  \left|\frac{\partial_x F(p)}{\partial_{y'} F(p)}\right| \leq 4 \min_{p \in Q}  \left|\frac{\partial_x F(p)}{\partial_{y'} F(p)}\right|\le C \delta^{-\eps},
\end{equation*}
and similarly for $\left|\partial_{x'} F(p)/\partial_{y'} F(p)\right|$ and $\left|\partial_{y} F(p)/\partial_{y'} F(p)\right|$. Thus

\begin{equation}\label{nablaGEstimate}
C^{-1}\delta^{\eps}\le  \max_{(x,x',y)\in I_1^Q\times I_2^Q\times I_3^Q}|\nabla G_{(y)}(x,x')| \leq 4 \min_{(x,x',y)\in I_1^Q\times I_2^Q\times I_3^Q}|\nabla G_{(y)}(x,x')|.
\end{equation}

Differentiating our expression for $\partial_x G$ and $\partial_{x'} G$ a second time and using the fact that all of the second derivatives of $F$ are bounded on $I_1^Q\times I_2^Q\times I_3^Q \times I_4$, we see that each of the second derivatives of $G$ has magnitude at most $\lesssim \delta^{-3\eps}$ on $I_1^Q\times I_2^Q\times I_3^Q$, and thus 
\begin{equation}\label{C2GEstimate}
\sup_{y\in I_3^Q} \Vert G_{(y)}\Vert_{C^2(I_1^Q\times I_2^Q)}\lesssim \delta^{-3\eps}.
\end{equation}

Define 
\begin{equation*}
\varphi_{(x,x')}(y):=\frac{\partial_{x'} G(x,x',y)}{\partial_x G(x,x',y)} =\tan(\theta_{(x,x')}(y)).
\end{equation*} 
We have
\begin{equation*}
\varphi_{(x,x')}(y)=\frac{\partial_{x'} F}{\partial_x F},
\end{equation*}
and 
\begin{align*}
\partial_y & \varphi_{(x,x')}(y)\\
&=\frac{(\partial_x F)(\partial_{y'} F)(\partial_{x'y}F)-(\partial_xF)(\partial_yF)(\partial_{x'y'}F)-(\partial_{x'}F)(\partial_{y'}F)(\partial_{xy}F)+(\partial_{x'}F)(\partial_yF)(\partial_{xy'}F)}{(\partial_{y'} F) (\partial_x F)^2}\\
&=\frac{H_F}{(\partial_{y'} F) (\partial_x F)^2}.
\end{align*}
Since $|H_F|\geq \delta^\eps$ on $I_1\times I_2\times I_3\times I_4$, for every $(x,x',y)\in I_1^Q\times I_2^Q\times I_3^Q$ and $y'=G(x,x',y)$, we have
\begin{equation}\label{phiBd}
|\varphi_{(x,x')}(y)| \lesssim \delta^{-\eps},
\end{equation}
and
\begin{equation}\label{phiPrimeBd}
|\partial_y \varphi_{(x,x')}(y)|\gtrsim  \delta^\eps.
\end{equation}
Since 
\begin{equation*}
\begin{split}
|\partial_y \varphi_{(x,x')}(y)|
&=|\partial_y(\tan(\theta_{(x,x')}))|\\
&=|(1+\tan^2(\theta_{(x,x')}(y))||\partial_y \theta_{(x,x')}(y)|,
\end{split}
\end{equation*}
we have
\begin{equation*}
|\partial_y \theta_{(x,x')}(y)|= \frac{|\partial_y \varphi_{(x,x')}(y)|}{1+(\varphi_{(x,x')}(y))^2},
\end{equation*}
and thus
\begin{equation}\label{derivativeThetaEstimate}
\inf_{(x,x',y)\in I_1^Q\times I_2^Q\times I_3^Q} |\partial_y \theta_{(x,x')}(y)|\gtrsim \delta^{3\eps}.
\end{equation}

Note that the estimates \eqref{nablaGEstimate}, \eqref{C2GEstimate}, and \eqref{derivativeThetaEstimate} are precisely the requirements needed to apply Lemma \ref{Shmcor} to $G$ and the sets $A_1^Q,A_2^Q,A_3^Q$, and $R = A_4$.
If $\eps>0$ is sufficiently small (with respect to $\alpha$ and $\kappa$), then there exists $\sigma=\sigma(\alpha,\kappa)>0$ so that
\begin{equation}\label{bdOnZcapQcapA}
\mathcal{E}_{\delta}\Big( \pi\big( (A_1^Q\times A_2^Q\times A_3^Q\times  A_4) \cap Z(F) \big) \Big) \lesssim \delta^{-3\alpha+\sigma}.
\end{equation}
We conclude that if $\eps=\eps(\alpha,\kappa)>0$ is selected sufficiently small, then \eqref{neededCoveringBdEachCube} holds.
\end{proof}

\section{The auxiliary function $K_P$}\label{auxiliaryFunctionSection}
In this section we will define and study the auxiliary function $K_P$ discussed in the introduction. To begin, we recall the following result from \cite{ER}. A detailed proof can be found in \cite[Lemma 10]{RaSh}.
\begin{lem}[Elekes and R\'onyai \cite{ER}]\label{special}
Let $U\subset\RR^2$ be a connected open set and let $P\colon U\to\RR$ be analytic. If $P_x$ or $P_y$ vanishes identically on $U$, then $P$ is an analytic special form, in the sense of Definition \ref{defnAnalyticSpecialForm}. If neither $P_x$ nor $P_y$ vanishes identically on $U$, then $P$ is an analytic special form if and only if $\partial_{xy}\left(\log\left(\frac{\partial_xP}{\partial_yP}\right)\right)$ vanishes everywhere on $U$ that it is defined (i.e.~everywhere $\partial_xP$ and $\partial_yP$ are non-zero). 
\end{lem}

\begin{defn}
If $P(x,y)$ is a real analytic function, define
\begin{equation*}
K_P(x,y) = \nabla P\wedge \nabla\Big(\frac{\partial_x P\partial_y P}{\partial_{xy}P}\Big).
\end{equation*}
\end{defn}
A computation shows that (formally),
\begin{equation}\label{relateKpLog}
\begin{split}
\big(\partial_x P \partial_y P\big)^2&\partial_{xy}\left(\log\left(\frac{\partial_xP}{\partial_yP}\right)\right)\\
&=(\partial_yP)^2\big( \partial_x P \partial_{xxy}P - \partial_{xx}P\partial_{xy}P\big)
-(\partial_x P)^2 \big( \partial_y P \partial_{xyy}P - \partial_{xy}P\partial_{yy}P\big)\\
&\qquad\qquad\qquad\qquad\qquad\qquad\qquad\qquad\qquad\qquad=\big(\partial_{xy}P\big)^2\Big(\nabla P \wedge \nabla \Big(\frac{\partial_x P \partial_y P}{\partial_{xy}P}\Big) \Big).
\end{split}
\end{equation}
If $P$ is smooth, then the middle expression is always well-defined. The first equality holds wherever the first expression is well-defined, while the second equality holds wherever the last expression is well-defined.

The identity \eqref{relateKpLog} and Lemma~\ref{special} has the following consequence.
\begin{lem}\label{whenIsPSpecialFormProp}
Let $U\subset\RR^2$ be a connected open set and let $P\colon U\to\RR$ be analytic (resp.~polynomial). Suppose that none of $\partial_x P$, $\partial_y P$, and $\partial_{xy}P$ vanishes identically on $U$. Then $K_P$ vanishes identically on $U$ if and only if $P$ is an analytic (resp.~polynomial) special form.
\end{lem}

\begin{lem}\label{countingNumberQuadruplesNablePWedgeP}

Let $0< \kappa\leq \alpha< 1$, let $U\subset\RR^2$ be a connected open set that contains $[0,1]^2$, and let $P\colon U \to\RR$ be analytic. Then there exists $\eps=\eps(\alpha,\kappa)>0$ and $\delta_0=\delta_0(\alpha,\kappa,P)$ such that the following holds for all $0<\delta\leq\delta_0$. Let $Q,Q^\prime\subset [0,1]^2$ be squares, and suppose that on $Q\cup Q'$ we have 
\begin{equation}
 |\partial_{xy}P|\geq\delta^\eps,\ |K_P|\geq\delta^\eps.
\end{equation}
Suppose furthermore that
\begin{equation}
\begin{split}
&\delta^{\eps}\leq\max_{(x,y)\in Q}|\partial_x P(x,y)|\leq 2\min_{(x,y)\in Q}|\partial_x P(x,y)|,\\
&\delta^{\eps}\leq\max_{(x,y)\in Q}|\partial_y P(x,y)|\leq 2\min_{(x,y)\in Q}|\partial_y P(x,y)|,
\end{split}
\end{equation}
and similarly for $Q^\prime$. 

Let $A,A',B,B'$ be sets with $A\times B\subset Q$ and $A'\times B'\subset Q^\prime.$ Suppose that for all intervals $J$ of length at least $\delta$, $A$ and $B$ satisfy the non-concentration condition.
\begin{equation}
\begin{split}
\mathcal{E}_{\delta}(A \cap J) &\le  |J|^\kappa\delta^{-\alpha-\eps},\\
\mathcal{E}_{\delta}(B \cap J) &\le  |J|^\kappa\delta^{-\alpha-\eps},
\end{split}
\end{equation}
and similarly for $A^\prime$ and $B^\prime$. Then
\begin{equation}\label{semiDiscretizedEnergy}
\mathcal{E}_{\delta}\big(\{ (x,x',y,y')\in A\times A'\times B\times B' \colon P(x,y) = P(x^\prime,y^\prime)\}\big)\lesssim \delta^{-3\alpha+\eps}.
\end{equation}
\end{lem}
\begin{proof}
Without loss of generality we can suppose that $A'$ and $B'$ are $\delta$-separated, and thus
\begin{equation}\label{cardApBp}
\#(A'\times B')\leq \delta^{-2\alpha-\eps}.
\end{equation}

Define $F(x,x',y,y') =P(x,y)-P(x',y')$ and recall Definition \ref{defnHFDefn}, which defines the auxiliary function $H_F$. By \eqref{defnHFForP} we have 
\begin{equation*}
\begin{split}
H_F(x,x',y,y') 
&=\partial_{xy}P(x,y)\partial_{x'y'}P(x^\prime,y^\prime)\Big(\frac{\partial_x P (x,y) \partial_y P(x,y)}{\partial_{xy}P(x,y)}-\frac{\partial_{x'}P(x^\prime,y^\prime)\partial_{y'}P(x^\prime,y^\prime)}{\partial_{x'y'}P(x^\prime,y^\prime)}\Big).
\end{split}
\end{equation*} 
Next we need to establish the estimate
\begin{equation}\label{coveringNumberH5small}
\mathcal{E}_{\delta}\big(\{ (x,x',y,y')\in A\times A'\times B\times  B^\prime \colon P(x,y) = P(x^\prime,y^\prime),\ |H_F(x,x',y,y')|<\delta^{4\eps/\kappa+3\eps}\}\big)\lesssim \delta^{-3\alpha+\eps}.
\end{equation}
Fix $x^\prime,y^\prime\in A^\prime \times B^\prime$. Use Theorem \ref{stratificationRealAnalSet} to cut $\{(x,y)\in Q\colon P(x,y) = P(x^\prime,y^\prime)\}$ into $O(1)$ smooth simple curves, where the implicit constant depends on $P$. Let $\gamma$ be one of these curves, and let $\gamma(t)\colon I\to Q$ be a unit speed parameterization of $\gamma$. For each $t\in I$, the unit tangent vector to $\gamma$ at $t$ is given by $\frac{\nabla P^\perp (\gamma(t))}{|\nabla P(\gamma(t))|}$. Thus for all $t\in I$ we have
\begin{equation*}
\begin{split}
\Big|\partial_t\Big(\frac{\partial_xP(\gamma(t))\partial_yP(\gamma(t))}{\partial_{xy}P(\gamma(t))}\Big)\Big|&=\Big|\frac{\nabla P^\perp}{|\nabla P|} \cdot \nabla\big(\frac{(\partial_x P)(\partial_y P)}{\partial_{xy}P}\big)(\gamma(t))\Big)\Big| \\
&= \frac{1}{|\nabla P|} \Big (\nabla P\wedge \nabla\big(\frac{(\partial_x P)(\partial_y P)}{\partial_{xy}P}\big)\Big)(\gamma(t))|\\
&= \frac{|K(\gamma(t))|}{|\nabla P|}\\
&\gtrsim \delta^{\eps}.
\end{split}
\end{equation*}
We conclude that for $(x^\prime,y^\prime)$ fixed, the set
\begin{equation*}
\Big\{(x,y)\in Q\colon P(x,y)=P(x^\prime,y^\prime),\ \Big|\frac{\partial_x P(x,y)\partial_y P(x,y)}{\partial_{xy}P(x,y)}-\frac{\partial_{x'}P(x',y')\partial_{y'}P(x',y')}{\partial_{x'y'}P(x',y')}\Big|\leq \delta^{4\eps/\kappa+\eps}\Big\}
\end{equation*}
is a union of $O(1)$ curve segments, each of length $\lesssim \frac{\delta^{4\eps/\kappa+\eps}}{\delta^{\eps}}=\delta^{4\eps/\kappa}$. Since $|\partial_{xy}P|\geq\delta^{\eps}$ on $Q\cup Q^\prime$, we conclude that the set
\begin{equation*}
\big\{(x,y)\in Q\colon P(x,y)=P(x^\prime,y^\prime),\ |H_F(x,x^\prime,y,y^\prime)|\leq \delta^{4\eps/\kappa+3\eps}\big\}
\end{equation*}
is a union of $O(1)$ curve segments, each of length $\lesssim \delta^{4\eps/\kappa}$. Thus
\begin{equation}\label{coveringNumberFixedxpyp}
\mathcal{E}_{\delta}\Big(\big\{(x,y)\in  A\times B \colon P(x,y)=P(x^\prime,y^\prime),\ |H(x,x^\prime,y,y^\prime)|\leq\delta^{4\eps/\kappa+3\eps}\big\}\Big) \lesssim \delta^{-\alpha-\eps} (\delta^{4\eps/\kappa})^{\kappa}=\delta^{-\alpha+3\eps}.
\end{equation}
This estimate holds for each $(x^\prime,y^\prime)\in A_1^\prime\times A_2^\prime$. \eqref{coveringNumberH5small} now follows from combining \eqref{cardApBp} and \eqref{coveringNumberFixedxpyp}.

Finally, we apply Lemma \ref{EnergyDispersionForF} to bound
\begin{equation}\label{energyBigH5}
\mathcal{E}_{\delta}\Big(\big\{ (x,x',y,y')\in A\times A'\times  B\times  B^\prime \colon P(x,y)=P(x^\prime,y^\prime),\ |H(x,x^\prime,y,y^\prime)|\geq\delta^{4\eps/\kappa+3\eps}\big\}\Big) \lesssim \delta^{-3\alpha+\eps}.
\end{equation}
Combining \eqref{coveringNumberH5small} and \eqref{energyBigH5}, we obtain \eqref{semiDiscretizedEnergy}, which completes the proof.
\end{proof}

\section{Proof of Theorem \ref{main-entropy-growth}} \label{mainThmProofsSectionOne}

At this point, we have assembled most of the necessary tools to prove Theorem \ref{main-entropy-growth}.  Lemma \ref{cubesWhereAnalyticFLarge} allows us to find a large square $Q$ where $P_x,P_y,P_{xy}$ and $K_P$ are large. Lemma~\ref{countingNumberQuadruplesNablePWedgeP} then says that there are few solutions to $P(a,b) = P(a',b')$ on this square. The final step is to use Cauchy-Schwarz to conclude that $P\big((A\times B) \cap Q\big)$ must be large. Here is a precise version of that statement
\begin{lem}\label{CSLem}
Let $Q\subset[0,1]$ be a square, and let $P\colon Q\to\RR$ be a smooth function with 
\[
|\partial_x P|\geq c,\quad |\partial_y P|\geq c\quad\textrm{on}\ Q.
\]
Let $X\subset Q$ be a union of squares of side-length $\delta$. Then
\[
\mathcal{E}_{\delta}(P(X))\gtrsim \frac{c\big(\mathcal{E}_{\delta}(X)\big)^2}{\mathcal{E}_{\delta}\big(\{(x,y,x',y')\in X^2\colon P(x,y) = P(x',y')\}\big)}.
\]
\end{lem}
\begin{proof}
Since $X$ is a union of squares of side-length $\delta$, we can select a $\delta$-separated set $\tilde X\subset X$ with $N_{\delta/2}(\tilde X) \subset X$ and $\#\tilde X = \mathcal{E}_{\delta}(X)$.

Define 
\begin{equation*}
\tilde P(x,y)=\frac{1}{2}c\delta\lfloor 2(c\delta)^{-1}P(x,y)\rfloor.
\end{equation*}
Let $\mathcal{Q}$ be the set of quadruples $(\tilde x, \tilde y, \tilde x', \tilde y')\in (\tilde X)^2$  with  $\tilde P(\tilde x, \tilde y) = \tilde P(\tilde x^\prime,\tilde y^\prime).$ For such a quadruple we have
\begin{equation}\label{tildePCloseToP}
|P(\tilde x, \tilde y) - P(\tilde x^\prime,\tilde y^\prime)|\leq \frac{1}{2}c\delta.
\end{equation}
Since $|\partial_y P|\geq c $ on $Q$, there exists $y^\prime$ with $|\tilde y'-y'|<\delta/2$ so that $P(\tilde x,\tilde y) = P(\tilde x^\prime,y^\prime)$. Since $N_{\delta/2}(\tilde X) \subset X$, we have $(\tilde x',y')\in X$. Since the quadruples in $\mathcal{Q}$ are $\delta$-separated, the corresponding quadruples $\{(\tilde x,\tilde y,\tilde x', y'\}$ are $\delta/2$ separated. In particular, we have 
\begin{equation}\label{boundSizeQVsPQuadruples}
\# \mathcal{Q}
\leq 16 \mathcal{E}_{\delta}\big(\{(x,y,x',y')\in X^2\colon P(x,y) = P(x',y')\}\big).
\end{equation}
By Cauchy-Schwarz, 
\begin{equation}\label{CSBoundTildeP}
\# \tilde P(\tilde X) \geq
(\# \tilde X)^2/ \#\mathcal{Q}.
\end{equation}
Finally, since $\tilde P$ takes values in $(\frac{1}{2}c\delta)\ZZ$ and in view of \eqref{tildePCloseToP},
 \begin{equation}\label{controlPAB}
\mathcal{E}_{\delta}(P(X))   \gtrsim c\big(\#\tilde P(\tilde X)\big).
\end{equation} 
The result now follows by combining \eqref{boundSizeQVsPQuadruples}, \eqref{CSBoundTildeP}, and \eqref{controlPAB}. 
\end{proof}

We are now ready to prove Theorem \ref{main-entropy-growth}. For the reader's convenience we will restate it here.
\begin{main-entropy-growthThm}
Let $0< \kappa\leq \alpha< 1$, let $U\subset\RR^2$ be a connected open set that contains $[0,1]^2$, and let $P\colon U \to\RR$ be analytic (resp.~polynomial). Then either $P$ is an analytic (resp.~polynomial) special form, or there exists $\eps = \eps(\alpha,\kappa)>0$ and $\eta = \eta(\alpha,\kappa,P)>0$ so that the following is true for all $\delta>0$ sufficiently small.

Let $A,B\subset[0,1]$ be sets with $\mathcal{E}_{\delta}(A)\geq\delta^{-\alpha}$, $\mathcal{E}_{\delta}(B)\geq\delta^{-\alpha},$ and suppose that for all intervals $J$ of length at least $\delta$, $A$ and $B$ satisfy the non-concentration conditions
\begin{equation*}\tag{\ref{nonConcentrationCond2}}
\begin{split}
\mathcal{E}_{\delta}(A \cap J) & \le  |J|^\kappa\delta^{-\alpha-\eta},\\
\mathcal{E}_{\delta}(B \cap J) & \le  |J|^\kappa\delta^{-\alpha-\eta}.
\end{split}
\end{equation*}
Then we have the entropy growth estimate
\begin{equation*}\tag{\ref{entropyGrowth}}
\mathcal{E}_{\delta}(P(A,B))\geq \delta^{-\alpha-\eps}.
\end{equation*}
If $P$ is a polynomial, then the quantity $\eta$ can be taken to depend only on $\alpha,\kappa$, and $\deg P$.
\end{main-entropy-growthThm}
\begin{proof}
First, since $|\nabla P|$ is bounded on $[0,1]^2$, we can replace $A$ and $B$ with their $\delta/2$ neighborhoods; doing this will only increase $\mathcal{E}_{\delta}(P(A,B))$ by a constant (depending on $P$, but independent of $\delta$) factor. Thus we can suppose that $A$ and $B$ are unions of $\delta$-intervals. 
We can assume that none of $\partial_x P$, $\partial_y P,$ $\partial_{xy}P$, or $K_P$ vanishes identically, since if any of these quantities vanish identically, then by Lemma~\ref{whenIsPSpecialFormProp}, $P$ is a special form and we are done. 

Let $C$ be the constant obtained by applying Lemma \ref{cubesWhereAnalyticFLarge} to the functions $\partial_x P$, $\partial_y P,$ $\partial_{xy}P$, and $K_P$. (Recall that in general $C$ depends on $P$, but if $P$ is a polynomial then $C$ depends only on $\deg P$.) We will apply this lemma with $w \sim C\eta/\kappa$ ($\eta>0$ will be chosen small enough that $w<\kappa/2$), and let $\mathcal{Q}$ be the resulting collection of cubes. With this choice of $w$, we have
\begin{equation*}
\mathcal{E}_{\delta}\Big( (A\times B) \cap \bigcup_{Q\in\mathcal{Q}}Q\Big)\geq \frac{1}{2}\mathcal{E}_{\delta}(A\times B).
\end{equation*}
By pigeonholing, there is a cube $Q\in\mathcal{Q}$ so that
\begin{equation*}
\mathcal{E}_{\delta}((A\times B)\cap Q) \gtrsim \delta^{4w/\kappa}\mathcal{E}_{\delta}(A\times B)\gtrsim \delta^{4C\eta/\kappa^2-2\alpha}.
\end{equation*}
Let $A_1\times B_1 = (A\times B)\cap Q$.  If $\eta = \eta(\alpha,\kappa,C)$ is selected sufficiently small, then by  Lemma~\ref{countingNumberQuadruplesNablePWedgeP}, there exists $\eps^\prime = \eps^\prime(\alpha,\kappa)>0$ so that 
\begin{equation}\label{quadruples}
{\mathcal E}_\delta\big(\{ ( x,  x',  y,  y')\in A_1^2\times B_1^2\colon P(x,y)=P(x',y')\}\big)\lesssim \delta^{-3\alpha+\eps^\prime}. 
\end{equation}

Applying Lemma \ref{CSLem} to $A_1\times B_1\subset Q$ with $c = \delta^w$, we conclude that
\[
\mathcal{E}_{\delta}(P(A,B))\geq \mathcal{E}_{\delta}(P(A_1,B_1))\gtrsim   \frac{\delta^{w} (\delta^{4C\eta/\kappa^2-2\alpha})^2}{\delta^{-3\alpha+\eps^\prime}}\geq \delta^{9C\eta/\kappa^2-\alpha-\eps'}.
\]
%
%
To complete the proof, select $\eta=\eta(\alpha,\kappa,C)$ sufficiently small so that $\eps^\prime\geq 18\eta C/\kappa^2$, and select $\eps = \eps^\prime/2$.
\end{proof}


\section{Proof of Theorem \ref{main-energy-dispersion}} \label{mainThmProofsSectionTwo}
For the reader's convenience we will restate Theorem \ref{main-energy-dispersion} here.

\begin{main-energy-dispersionThm}
Let $0< \kappa\leq \alpha< 1$, let $U\subset\RR^2$ be a connected open set that contains $[0,1]^2$, and let $P\colon U \to\RR$ be analytic (resp.~polynomial). Then either $P$ is an analytic (resp.~polynomial) special form, or there exists $\eps = \eps(\alpha,\kappa,P)>0$ and $\eta = \eta(\alpha,\kappa,P)>0$ so that the following is true for all $\delta>0$ sufficiently small.

Let $A,B\subset[0,1]$ and suppose that for all intervals $J$ of length at least $\delta$, $A$ and $B$ satisfy the non-concentration condition
\begin{equation*}\tag{\ref{nonConcentrationCond}}
\begin{split}
\mathcal{E}_{\delta}(A \cap J) &\le  |J|^\kappa\delta^{-\alpha-\eta},\\
\mathcal{E}_{\delta}(B \cap J) &\le  |J|^\kappa\delta^{-\alpha-\eta}.
\end{split}
\end{equation*}
Then we have the energy dispersion estimate
\begin{equation*}\tag{\ref{energyDispersionEqn}}
\mathcal{E}_{\delta}\big(\{(x,x',y,y')\in A^2\times B^2 \colon P(x,y) = P(x^\prime,y^\prime) \}\big)\leq \delta^{-3\alpha+\eps}.
\end{equation*}
If $P$ is a polynomial, then the quantities $\eps$ and $\eta$ can be taken to depend only on $\alpha,\kappa$, and $\deg P$.
\end{main-energy-dispersionThm}
\begin{proof}
The idea of the proof is as follows: If $\partial_x P$, $\partial_y P,$ $\partial_{xy}P$, or $K_P$ vanishes identically, then $P$ is a special form and we are done. Next suppose this does not occur. The subset of $A\times B$ where the quantities $\partial_x P$, $\partial_y P,$ $\partial_{xy}P$, and $K_P$ are small has small $\delta$-covering number. We will call this the ``bad'' region, and the remainder of $A\times B$ we will call the ``good'' region. Since the bad region has small $\delta$-covering number, we will be able to prove that there are few quadruples $(x,x',y,y')$ contributing to \eqref{energyDispersionEqn} for which $(x,y)$ or $(x',y')$ is in the bad region. It remains to count the number of quadruples contributing to \eqref{energyDispersionEqn} for which both $(x,y)$ and $(x',y')$ are in the good region. Our argument here will be similar to the proof of Theorem~\ref{main-entropy-growth}. 
The good region can be efficiently covered by a small number of squares. We can use Lemma~\ref{countingNumberQuadruplesNablePWedgeP} to bound the number of quadruples $(x,x',y,y')$ contributing to \eqref{energyDispersionEqn} where $(x,y)$ and $(x',y')$ are contained in (possibly different) squares inside the good region. Summing this bound over all pairs of good squares completes the argument. We now turn to the details.

To begin, we can suppose that $A$ and $B$ are unions of intervals of the form $[n\delta, (n+1)\delta)$ with $n\in\ZZ$. Indeed, we can replace $A$ (resp. $B$) by the union of all intervals of this form that intersect $A$. Doing so weakens Inequality \eqref{nonConcentrationCond} by a harmless multiplicative factor, and does not decrease the LHS of \eqref{energyDispersionEqn}. 

Let $P\colon U\to\RR$ be analytic (resp.~polynomial). If $\partial_x P$, $\partial_y P,$ or $\partial_{xy}P$ vanish identically on $U$ then $P$ is an analytic (resp.~polynomial) special form and we are done. If $K_P$ vanishes identically on $U$, then by Lemma \ref{whenIsPSpecialFormProp}, $P$ is an analytic (resp.~polynomial) special form and we are done.

Next, suppose that none of these quantities vanish identically. Let $C$ be the constant obtained by applying Lemma \ref{cubesWhereAnalyticFLarge} to the functions $\partial_x P$, $\partial_y P,$ $\partial_{xy}P$, and $K_P$. (Recall that in general $C$ depends on $P$, but if $P$ is a polynomial then $C$ depends only on $\deg P$.) We will apply Lemma~\ref{cubesWhereAnalyticFLarge} with a choice of $w=w(\alpha,\kappa)>0$ to be chosen later, and let $\mathcal{Q}$ be the resulting collection of cubes. We have
\begin{equation}\label{numberOfCubes}
\#\mathcal{Q} \lesssim \delta^{-4w/\kappa},
\end{equation}
\begin{equation}
|\partial_{xy}P|,\ |K_P|\ \geq \delta^w\quad\textrm{on each cube}\ Q\in\mathcal{Q},
\end{equation}
and for each cube $Q\in\mathcal{Q}$, there are numbers $v_{Q,x},v_{Q,y}\geq \delta^w$ so that
\begin{equation}\label{sizeOfP1212Nabla}
\begin{split}
&v_{Q,x}\leq |\partial_x P|< 4v_{Q,x},\\
&v_{Q,y}\leq |\partial_y P|< 4v_{Q,y}
\end{split}
\end{equation}
on $Q$. Finally, we have 
\begin{equation*}
\mathcal{E}_{\delta}\Big( (A\times B)\ \backslash \bigcup_{Q\in\mathcal{Q}}Q\Big)\lesssim \delta^{-2\alpha + \frac{w\kappa}{C}-3\eta}.
\end{equation*}

Let $W = [0,1]^2 \backslash \bigcup_{Q\in\mathcal{Q}}Q$. The following claim allows us to control the number of quadruples $P(x,y)=P(x',y')$ where either $(x,y)$ or $(x',y')$ is contained in $W$. 
\begin{clm}
If $\delta>0$ is chosen sufficiently small (depending $P$), then
\begin{equation}\label{quadruplesNotCoveredBySquares}
\begin{split}
\mathcal{E}_{\delta}\Big((x,x',y,y')\colon (x,y)\in W\cap (A\times B)\ \textrm{or}\ (x',y')\in W\cap (A\times B),\ & P(x,y) = P(x',y')\Big)\\
&\lesssim \delta^{-3\alpha+\frac{w\kappa^2}{2C^2}-2\eta }.
\end{split}
\end{equation}
\end{clm}
\begin{proof}
First, since $A$ and $B$ are unions of intervals of the form $[n\delta, (n+1)\delta)$, the estimate \eqref{quadruplesNotCoveredBySquares} is equivalent to counting the number of pairs of squares
\begin{equation}\label{pairsOfSquares}
\{ (S, S') \colon \exists\ (x,y)\in S,\ (x', y')\in S'\ \textrm{s.t.}\ (x,y)\in W\ \textrm{or}\ (x',y')\in W,\ P(x,y) = P(x',y')\},
\end{equation}
where $S$ and $S'$ are squares of the form $[n\delta, (n+1)\delta)\times [m\delta, (m+1)\delta)$ that are contained in $A\times B$. For each square $S$ of the above type, define
\begin{equation*}
S_{\max} = \sup_{(x,y)\in S}P(x,y),\quad S_{\min} = \inf_{(x,y)\in S}P(x,y),\quad S_{\operatorname{var}} = S_{\max}-S_{\min}.
\end{equation*}
Since $P$ is continuous and $S$ is connected, by the intermediate value theorem we have that
\begin{equation*}
(S_{\min},S_{\max}) \subset \{P(x,y)\colon (x,y)\in S\}\subset [S_{\min},S_{\max}].
\end{equation*}
Note that by the mean value theorem, there is a point $(x,y)\in S$ so that 
\begin{equation*}
|\nabla P(x,y)| \geq \frac{S_{\operatorname{var}}}{\operatorname{diam}(S)} = (\sqrt 2 \delta)^{-1} S_{\operatorname{var}}.
\end{equation*}
Note as well that since $|\nabla^2 P|$ is bounded on $[0,1]^2$, we have that $|\nabla P|$ can only change by $\lesssim \delta$ on $S$. 
Informally, this means that if $|\nabla P|$ is small somewhere on $S$, then it is small everywhere on $S$, and $S_{\operatorname{var}}$ must be small. 

Next, let $w_1>0$ be a small parameter that we will specify below. We will count the number of pairs of squares $(S,S')$ that satisfy:
\begin{itemize}
\item $P(x,y) = P(x',y')$ for some $(x,y)\in S,\ (x',y')\in S'$.
\item $|\nabla P(x,y)|<\delta^{w_1}$ for some $(x,y)\in S$.
\item $S_{\operatorname{var}}<2 S'_{\operatorname{var}}$.
\end{itemize}
Note that if $|\nabla P(x,y)|<\delta^{w_1}$ for some $(x,y)\in S$, then provided $\delta$ is chosen sufficiently small (depending on $P$), we have $|\nabla P|\leq (1.1)\delta^{w_1}$ everywhere on $S$ (the constant $1.1$ is not important; the statement is true for any constant larger than 1). 

To count the number of pairs $(S,S')$ satisfying the three properties listed above, we begin by observing that by Theorem \ref{Lojasiewicz} and Corollary \ref{coveringNbhAnalyticSet}, 
after possibly increasing the constant $C$, there are $\lesssim \delta^{-2\alpha + \frac{w_1\kappa}{C}-3\eta}$ choices for $S$. Fix such a square $S$ and define the curves
\begin{equation*}
\begin{split}
\alpha_1(S) &= \{ (x',y')\in [0,1]^2\colon P(x',y') = S_{\min} \},\\
\alpha_2(S) &= \{ (x',y')\in [0,1]^2\colon P(x',y') = (S_{\min}+S_{\max})/2 \},\\
\alpha_3(S) &= \{ (x',y')\in [0,1]^2\colon P(x',y') = S_{\max} \}.
\end{split}
\end{equation*}
Next, let $S'$ be a square so that $P(x,y) = P(x',y')$ for some $(x,y)\in S,\ (x',y')\in S'$ and $S_{\operatorname{var}}< 2 S'_{\operatorname{var}}$. Then
\begin{equation*}
S_{\min}\leq P(x,y) = P(x', y') \leq S_{\max},
\end{equation*}
and the interval $(S'_{\min}, S'_{\max})$ has length greater than $S_{\operatorname{var}}/2$. Thus $(S'_{\min}, S'_{\max})$ must contain at least one of the points $S_{\min}$, $(S_{\min}+S_{\max})/2$, or $S_{\max}$, and this implies that $S'$ must intersect at least one of the three curves $\alpha_1(S),\alpha_2(S),$ or $\alpha_3(S)$. By Corollary \ref{coveringNbhAnalyticSet} (with $s=\delta$), each of these three curves intersect $\lesssim\delta^{-\alpha-\eta}$ cubes from $A\times B$; we conclude that the number of pairs $(S,S')$ satisfying the above requirements is $\lesssim \delta^{-3\alpha + \frac{w_1\kappa}{C}-4\eta}$. An identical argument allows us to bound the number of cubes $(S,S')$ satisfying the three properties listed above, with the roles of $S$ and $S'$ reversed. 

Finally, we will remove the requirement that $S_{\operatorname{var}}<2 S'_{\operatorname{var}}$. Suppose that the pair of squares $(S,S')$ satisfies the first two items listed above, but that $S_{\operatorname{var}}\geq 2 S'_{\operatorname{var}}$, then $|\nabla P|<(1.1)\delta^{w_1}$ everywhere on $S$, and this implies $S_{\operatorname{var}}\leq (1.1)\sqrt 2\delta^{1+w_1}<2\delta^{1+w_1}$, and hence $S'_{\operatorname{var}}< \delta^{1+w_1}$. But this means that the pair $(S',S)$ satisfies all three of the items listed above, and our  previous argument has already bounded the number of pairs of this form. 

To summarize,
\begin{equation}
\begin{split}
\# \{ (S,S') \colon&\ P(x,y) = P(x',y')\ \textrm{for some}\ (x,y)\in S,\ (x',y')\in S',\\
& |\nabla P(x,y)|\leq \delta^{w_1}\ \textrm{for some}\ (x,y)\in S\ \textrm{or}\ |\nabla P(x',y')|\leq \delta^{w_1}\ \textrm{for some}\ (x',y')\in S' \} \\
& \lesssim \delta^{-3\alpha + \frac{w_1\kappa}{C}-4\eta}.
\end{split}
\end{equation}

Next we will consider pairs $(S,S')$ in \eqref{pairsOfSquares} where $|\nabla P|\geq\delta^{w_1}$ on $S$ and on $S'$. 
As discussed above, assuming that $S\cap W\neq\emptyset$, there are $\lesssim  \delta^{-2\alpha + \frac{w\kappa}{C}-3\eta}$ choices for $S$. Fix such a square $S$, and let $z_1(S),\ldots z_N(S)$ be evenly spaced points in $[S_{\min},S_{\max}]$, with $z_1(S) = S_{\min}$ and $z_N(S) = S_{\max}$. We will choose $N = K\delta^{-w_1}$, where $K$ is a constant that depends only on $P$. For each index $i=1,\ldots,N$, define the curve
\begin{equation*}
\alpha_i(S) = \{(x',y')\in [0,1]^2 \colon P(x',y') = z_i(S)\}.
\end{equation*}
Since $|\nabla P|\lesssim 1$ on $[0,1]^2$, we have $S_{\operatorname{var}}\lesssim \delta$, and thus every number between $S_{\min}$ and $S_{\max}$ must be $\lesssim N^{-1}\delta= K^{-1}\delta^{1+w_1}$ close to some $z_i(S)$. 
Let $S'$ be a square so that $P(x,y) = P(x',y')$ for some $(x,y)\in S$ and some $(x',y')\in S^\prime$. Since $|\nabla P|\geq\delta^{w_1}$ on $S'$ we must have $S'_{\operatorname{var}}\gtrsim \delta^{1+w_1}$. 
In particular, if $K$ is chosen sufficiently large (depending on $P$), then the interval $(S'_{\min},S'_{\max})$ must contain one of the points $z_i(S)$, and thus the square $S'$ must intersect at least one of the curves $\alpha_i(S)$. By Corollary \ref{coveringNbhAnalyticSet}, each curve $\alpha_i(S)$ can intersect $\lesssim \delta^{-\alpha-\eta}$ cubes $S'$; we conclude that there are $\lesssim \delta^{-3\alpha + \frac{w\kappa}{C} - w_1-4\eta}$ pairs $(S,S')$ in \eqref{pairsOfSquares} where $|\nabla P|\geq\delta^{w_1}$ on $S$ and on $S'$. 

In summary,
\begin{equation}
\#\eqref{pairsOfSquares} \lesssim \delta^{-3\alpha}(\delta^{\frac{w_1\kappa}{C}-4\eta} + \delta^{\frac{w\kappa}{C} - w_1- 4\eta}). 
\end{equation}
To finish, we select $w_1=\frac{w\kappa}{2C}.$
\end{proof}

We are now ready to continue with the proof of Theorem \ref{main-energy-dispersion}. Let $\eps^\prime = \eps^\prime(\alpha,\kappa)>0$ be the constant from Lemma \ref{countingNumberQuadruplesNablePWedgeP}. If $w$ and $\eta$ are selected sufficiently small (depending on $\alpha,$ $\kappa,$ and $C$), then the hypotheses of Lemma \ref{countingNumberQuadruplesNablePWedgeP} are met for each pair of cubes $Q,Q^\prime\in\mathcal{Q}$. We conclude that
\begin{equation}\label{coveringQuadruplesInOnePairOfCubes}
\mathcal{E}_{\delta}\big(\{ (x,x',y,y')\colon (x,y)\in Q,\ (x',y')\in Q',\  P(x,y) = P(x',y')\}\big)\lesssim \delta^{-3\alpha+\eps^\prime}.
\end{equation} 
Combining \eqref{numberOfCubes} and \eqref{coveringQuadruplesInOnePairOfCubes}, and noting that there are at most  $(\#\mathcal{Q})^2$ pairs of cubes, we have
\begin{equation}\label{quadruplesCoveredBySquares}
\mathcal{E}_{\delta}\Big((x,x',y,y')\colon (x,y),(x',y')\in (A\times B)\backslash W,\ P(x,y) = P(x',y')\Big)\lesssim\delta^{-3\alpha+\eps^\prime -4w/\kappa}.
\end{equation}

Combining \eqref{quadruplesNotCoveredBySquares} and \eqref{quadruplesCoveredBySquares}, we have
\begin{equation}
\mathcal{E}_{\delta}\Big((x,x',y,y')\in A^2\times B^2\colon P(x,y) = P(x',y')\Big)\lesssim \delta^{-3\alpha}(\delta^{\frac{w\kappa^2}{2C^2}-4\eta } + \delta^{\eps^\prime -4w/\kappa}).
\end{equation}
To finish, choose $w$ sufficiently small so that $\eps^\prime -4w/\kappa\geq \eps^\prime/2$. 
Choose $\eta$ sufficiently small so that $\frac{w\kappa^2}{2C^2}-4\eta\geq \frac{w\kappa^2}{4C^2}$. 
Finally, choose $\eps>0$ sufficiently small so that $\eps\leq \min(\eps^\prime/2, \frac{w\kappa^2}{4C^2})$.
\end{proof}

\section{Dimension expansion}\label{dimExpansionSec}
In this section we will prove Theorem \ref{dimension-expander}.
The basic idea is we select a square $Q$ so that the sets $A',B'$ defined by $A'\times B' = (A\times B)\cap Q$ have large dimension, and $\partial_xP,\ \partial_yP, \partial_{xy}P$, and $K_P$ are large on $Q$. We then discretize and apply Lemma \ref{countingNumberQuadruplesNablePWedgeP}.

\subsection{Finding a good square}
\begin{defn}
Let $A\subset\RR$, let $x\in\RR$, and let $\beta>0$. We say that $A$ has {\it local dimension} $\geq\beta$ at $x$ if $\dim(A\cap U)\geq\beta$ for every neighborhood $U$ of $x$. Otherwise we say $A$ has local dimension $<\beta$ at $x$. 
\end{defn}

\begin{lem}\label{localDimMaximal}
Let $A\subset\RR$. Then for each $\eps>0$, there is at least one point $x\in \RR$ where $A$ has local dimension $\geq\dim(A)-\eps$. 
\end{lem}
\begin{proof}
Suppose not. For each $x\in \RR$, let $U_x$ be a neighborhood of $x$ with $\dim(A\cap U_x)\leq \dim(A)-\eps$. Let $\mathcal{U}\subset \{U_x\colon x\in A\}$ be a countable sub-cover of $\RR$. Then $\dim(A) = \max_{U\in\mathcal{U}}\dim(A\cap U)\leq \dim(A)-\eps$, which is impossible.
\end{proof}

\begin{cor}\label{ManyPointsAchieveLocalDim}
Let $A\subset\RR$, let $\eps>0$, and let $B$ be the set of points $x\in\RR$ where $A$ has local dimension $\ge \dim(A)-\eps$. Then either $A\setminus B=\emptyset$ or $\dim A\setminus B < \dim A$. In particular, if $\dim A>0$ then 
there are infinitely many points $x\in\RR$ where $A$ has local dimension $\geq\dim(A)-\eps$. 
\end{cor}

\begin{lem}\label{findGoodPoint}
Let $f$ be a function that is analytic and not identically zero on an open neighborhood of $[0,1]^d$, let $\eps>0$, and let $A_1,\ldots,A_d\subset [0,1]$ have positive Hausdorff dimension. Then there is a point $p=(p_1,\ldots,p_d)\in (0,1)^d\backslash Z(f)$ so that for each index $i$, $A_i$ has local dimension $\geq\dim(A_i)-\eps$ at $p_i$.
\end{lem}
\begin{proof}
We will prove the result by induction on $d$. When $d=1$ the result follows from Corollary~\ref{ManyPointsAchieveLocalDim} and the fact that since $f$ is analytic and not identically zero on a neighborhood of $[0,1]$, $Z(f)\cap [0,1]$ is finite. Now suppose the result is true for $d-1$ and let $f,$ $\eps$, and $A_1,\ldots,A_d$ be as in the statement of the lemma. Since $f$ is analytic and not identically zero on a neighborhood of $[0,1]^d$, there are finitely many points $p_d\in [0,1]$ so that the (truncated) hyperplane $\{x_d=p_d\}\cap [0,1]^d$ is contained in $Z(f) \cap [0,1]^d$. Use Corollary \ref{ManyPointsAchieveLocalDim} to select a point $p_d\in (0,1)$ so that $A_d$ has local dimension $\geq\dim(A_d)-\eps$ at $p_d$, and the truncated hyperplane $\{x_d=p_d\}\cap [0,1]^d$ is not contained in $Z(f)\cap [0,1]^d$. Then the function $(x_1,\ldots,x_{d-1})\mapsto f(x_1,\ldots,x_{d-1},p_d)$ is analytic and not identically zero on a neighborhood of $[0,1]^{d-1}$. We now find the point $p=(p_1,\ldots,p_{d-1},p_d)$ by applying the induction hypothesis to this function and the sets $A_1,\ldots,A_{d-1}$.
\end{proof}

We are now ready to prove Theorem \ref{dimension-expander}. For the reader's convenience we will reproduce it here.
\begin{dimension-expanderThm}
Let  $U\subset\RR^2$  be a connected open set that contains $[0,1]^2$ and let $P\colon U \to\RR$ be analytic (resp.~polynomial). Then either $P$ is an analytic (resp.~polynomial) special form, or for every $0< \alpha< 1$, there exists $\eps=\eps(\alpha)>0$ so that the following holds. Let $A,B\subset [0,1]$ be Borel sets with Hausdorff dimension at least $\alpha$. Then 
\begin{equation*}\tag{\ref{dimensionGrowth}}
\dim P(A,B) \ge \alpha+\eps.
\end{equation*}  
\end{dimension-expanderThm}
\begin{proof}
Let $\eps=\eps(\alpha)$ and $\delta_0=\delta_0(\alpha,P)$ be the quantities obtained by applying Lemma \ref{countingNumberQuadruplesNablePWedgeP} to $P$ with parameters $\alpha$ and $\kappa = \alpha/2$. We will prove that 
\begin{equation}\label{deimPAB}
\dim(P(A,B))>\alpha+\eps/2.
\end{equation} 

Suppose that $P$ is not a special form and \eqref{deimPAB} is false. Without loss of generality we can suppose that $P(A,B)\subset [0,1]$. Define 
\begin{equation*}
f= (\partial_x P)(\partial_y P)(\partial_{xy}P)^2(K_P).
\end{equation*}
By Lemma~\ref{whenIsPSpecialFormProp}, $f$ does not vanish identically on $[0,1]^2$. Note that initially $f$ is not defined on $Z(\partial_{xy}P)$, but if we define $f$ to be $0$ on this set, then $f$ is analytic on a neighborhood of $[0,1]^2$, so we can use Lemma \ref{findGoodPoint} to select a point $p=(p_1,p_2) \in (0,1)^2\backslash Z(f)$ so that $A$ and $B$ have local dimension $\geq \alpha-\eps/20$ at $p_1$ and $p_2$, respectively. Let $I_0,J_0\subset [0,1]$ be closed intervals containing $p_1$ and $p_2$ respectively so that $(I_0\times J_0)\cap Z(f)=\emptyset$ and $p$ is contained in the interior of $I_0\times J_0$. Then there is a number $c>0$ so that 
\begin{equation}\label{lowerBoundKeyQuantities}
|\partial_x P(x,y)|\geq c,\ |\partial_y P(x,y)|\geq c,\ |\partial_{xy}P(x,y)|\geq c,\ |K_p(x,y)|\geq c\quad\forall\ (x,y)\in I_0\times J_0. 
\end{equation}
%
Define $A^\prime = A\cap I_0$ and $B^\prime = B\cap J_0$. Let $C = P(A',B')$.

Since $\dim(A')\geq \alpha-\eps/20$ and $\dim(B')\geq \alpha-\eps/20$, by Frostman's lemma (see e.g.~\cite[Theorem 8.8]{Mat}), there are Borel  measures $\mu$ and $\nu$, supported on $A'$ and $B'$ respectively, with $\mu(A')>0,$ and $\nu(B')>0$, so that
\begin{equation}\label{FrostmanConditionIneq}
\begin{split}
&\mu(A\cap I)\leq |I|^{\alpha-\eps/10}\ \textrm{for every interval}\ I,\\
&\nu(B\cap J)\leq |J|^{\alpha-\eps/10}\ \textrm{for every interval}\ J.
\end{split}
\end{equation}
In particular, $\mu$ and $\nu$ have no atoms. Let $\lambda$ be the pushforward of $\mu\times\nu$ by $P$, i.e. for each Borel set $U \subset \RR$, we have 
\begin{equation*}
\lambda(U) = (\mu\times\nu) (\{(x,y)\in A\times B: P(x,y) \in U\}).
\end{equation*}
Observe that $\lambda$ is a measure supported on $C$, and $\lambda(C)>0$. Let $\delta_1\leq\delta_0$ be a number of the form $2^{-k_1}$ for some positive integer $k_1$. In what follows, all implicit constants will be independent of $\delta_1$. If $\delta_1>0$ is selected sufficiently small (depending only on $\alpha$, 
then we will eventually show that \eqref{deimPAB} must hold.

Let $\{[x_i, x_i+r_i]\}$ be a covering of $C$ by intervals of length at most $\delta_1$, with
\begin{equation}\label{sumOfRi}
\sum_i r_i^{\alpha+(3/5)\eps}<1.
\end{equation}
Such a covering must exist, since by assumption we have $\dim(C)\leq \alpha+\eps/2$. Without loss of generality we can suppose that each interval is of the form $[n2^{-k},(n+1)2^{-k}]$ for some $k\geq k_1$ and some integer $1\leq n< 2^{-k_1}$. 

For each $k\geq k_1$, let $m_k = \sum \lambda([x_i,x_i+r_i])$, where the sum is taken over all intervals of length $r_i=2^{-k}$. Since $\sum_{k\geq k_1} m_k \geq \lambda(C)>0$ and $\sum_{k\geq k_1}k^{-2}\le \pi^2/6$, there must exist an index $k_2\geq k_1$ with $m_{k_2} \gtrsim \lambda(C)k_2^{-2}$. Define $\delta = 2^{-k_2}$, so $0<\delta\leq\delta_1$.

At this point it will be helpful to introduce some additional notation. We say $X\lessapprox Y$ if there is an absolute constant $K$ so that $A\leq K |\log\delta|^K B$. If $A\lessapprox B$ and $B\lessapprox A$, we say $A\approx B$. In the arguments that follow, we can always take $K\leq 100$. With this notation we have $1\lessapprox m_{k_2}$. 

Let $C^\prime=\bigcup [x_i,x_i+r_i]$, where the union is taken over all intervals of length $r_i=\delta$ in the covering. Then $\lambda(C^\prime)=m_{k_2}\gtrapprox 1$, and $C^\prime$ is a union of interior-disjoint intervals of the form $[n\delta,(n+1)\delta]$, where $1\leq n<\delta^{-1}$ is an integer. After dyadic pigeonholing, we can select a set $C^{\prime\prime}\subset C^{\prime}$ that is again a union of $\delta$-intervals with
\begin{equation}\label{massOfCpp}
\lambda(C^{\prime\prime})\gtrapprox 1,
\end{equation} 
and each $\delta$-interval $I\subset C^{\prime\prime}$ has measure $\lambda(I)\sim \lambda(C^{\prime\prime})(\delta|C^{\prime\prime}|^{-1})$. 
Note that 
\begin{equation}\label{sizeOfCalC}
\delta^{1-\alpha+\eps/10}\lessapprox |C^{\prime\prime}|\le \delta^{1-\alpha-(3/5)\eps}.
\end{equation}
The lower bound on $|C''|$ follows from \eqref{FrostmanConditionIneq} and \eqref{massOfCpp}, while the upper bound follows from \eqref{sumOfRi}. 

Cover $I_0\times J_0$ by squares of the form $[n\delta,(n+1)\delta]\times[m\delta,(m+1)\delta]$, and let $\mathcal{Q}_0$ be the set of squares for which $P(Q)\cap C^{\prime\prime}\neq\emptyset$. We have
\begin{equation*}
\sum_{Q\in\mathcal{Q}_0}(\mu\times\nu)(Q)\geq \lambda(C^{\prime\prime})\gtrapprox 1. 
\end{equation*}
By dyadic pigeonholing, there is a set $\mathcal{Q}\subset\mathcal{Q}_0$ and numbers $m_x,m_y$ so that
\begin{equation}\label{QCapturesMostMass}
\sum_{Q\in\mathcal{Q}}(\mu\times\nu)(Q)\gtrapprox 1,
\end{equation}
$m_x\leq (\mu(\pi_x(Q))\leq 2m_x$, and $m_y\leq (\mu(\pi_x(Q))\leq 2m_y$ for each $Q\in\mathcal{Q}$. In particular,
\begin{equation*}
\#\mathcal{Q} \approx (m_xm_y)^{-1}.
\end{equation*}

Our next goal is to show that $\mathcal{Q}$ resembles (a dense subset of) a Cartesian product. Let $A_1=\bigcup_{Q\in\mathcal{Q}}\pi_x(Q)$ and let $B_1 = \bigcup_{Q\in\mathcal{Q}}\pi_y(Q).$ We have that $A_1$ and $B_1$ are interior-disjoint unions of $\delta$-intervals contained in $I_0$ and $J_0$, respectively; $m_x\leq \mu(I)<2m_x$ for each $\delta$ interval $I\subset A_1$; $m_y\leq \nu(J)<2m_y$ for each $\delta$ interval $J\subset B_1$; and $Q\subset A_1\times B_1$ for each square $Q\in\mathcal{Q}$. 

Clearly $A_1\times B_1$ is a union of $\delta$-squares, and there are at least $\#\mathcal{Q}\approx (m_xm_y)^{-1}$ squares in this union. On the other hand, we have $m_xm_y \leq (\mu\times\nu)(Q)< 4 m_x m_y$ for each $\delta$-square $Q\subset A_1\times B_1$, so $A_1\times B_1$ is a union of $\lesssim (m_xm_y)^{-1}$ $\delta$-squares. 

Our final task is to show that $A_1$ and $B_1$ each have size  roughly $\delta^{1-\alpha}$. First, since $\mathcal{Q}$ contains $\gtrapprox (m_xm_y)^{-1}$ squares and $A_1\times B_1$ contains $\lesssim (m_xm_y)^{-1}$ squares, there exists $y_0\in J_0$ so that the line $y=y_0$ intersects $\gtrapprox  \delta^{-1}|A_1|$ squares from $\mathcal{Q}$. By \eqref{sizeOfCalC} and \eqref{lowerBoundKeyQuantities}, we have
\begin{equation*}
\delta^{1-\alpha-\eps/10}\gtrapprox |C''|\geq \bigcup_{\substack{Q\in\mathcal{Q}\\
Q\cap \{y =y_0\}\neq\emptyset}} P(Q) \gtrsim \delta\#(\{Q\in\mathcal{Q}\colon Q\cap \{y=y_0\}\neq\emptyset\})\gtrapprox |A_1|,
\end{equation*}
and an identical argument shows that
\begin{equation*}
|B_1|\lessapprox  \delta^{1-\alpha-\eps/10}.
\end{equation*}

But \eqref{FrostmanConditionIneq} and \eqref{QCapturesMostMass} implies that $|A_1|,|B_1|\gtrapprox \delta^{1-\alpha+\eps/10}.$ 
We conclude that
\begin{equation}\label{sizeOfA0B0}
\delta^{1-\alpha+\eps/10}\lessapprox |A_1|,|B_1|\lessapprox  \delta^{1-\alpha-\eps/10}.
\end{equation}

Next, if $Q,Q^\prime\in \mathcal{Q}$ are two squares that intersect a common $\delta$-interval from $C''$, and if $c>0$ is the constant from \eqref{lowerBoundKeyQuantities}, then there are points $(x,y)\in N_{c^{-1}\delta}(Q)$ and $(x',y')\in N_{c^{-1}\delta}(Q')$ so that $P(x,y) = P(x',y')$. In particular, if we define $A_2 = N_{c^{-1}\delta}A_1$ and $B_2 = N_{c^{-1}\delta}B_1$, then
\begin{equation*}
\begin{split}
\mathcal{E}_{\delta}\big(& \{x,x',y,y'\in A_2^2\times B_2^2\colon P(x,y) = P(x',y')\big)\\
&\gtrsim \#\{ (Q,Q') \in \mathcal{Q}^2\colon Q\ \textrm{and}\ Q'\ \textrm{intersect a common $\delta$-interval from}\ C^{\prime\prime}\}\\
&\gtrsim (\#\mathcal{Q})^2(\delta|C''|^{-1})\\
&\gtrapprox \delta^{-3\alpha + (4/5)\eps},
\end{split}
\end{equation*}
i.e. there is an absolute constant $K_1$ so that
\begin{equation}\label{lowerBoundNumberQuadruples}
\mathcal{E}_{\delta}\big( \{x,x',y,y'\in A_2^2\times B_2^2\colon P(x,y) = P(x',y')\big)\geq K_1^{-1}|\log\delta|^{-K_1}\delta^{-3\alpha + (4/5)\eps}.
\end{equation}

Finally, by \eqref{FrostmanConditionIneq}, the sets $A_2$ and $B_2$ obey the hypotheses of Lemma \ref{countingNumberQuadruplesNablePWedgeP} with parameters $\alpha$ and $\kappa = \alpha/2$. Since $\delta\leq\delta_0$, there exists a constant $K_2$ (which may depend on $\alpha$) so that 
\begin{equation}\label{upperBoundNumberQuadruples}
\mathcal{E}_{\delta}\big( \{x,x',y,y'\in A_2^2\times B_2^2\colon P(x,y) = P(x',y')\big)\leq K_2 \delta^{-3\alpha + \eps}.
\end{equation}
Thus if $\delta_1$ is selected sufficiently small (depending on $K_1$ and $K_2$, which in turn depend only on $\alpha$), then \eqref{upperBoundNumberQuadruples} contradicts \eqref{lowerBoundNumberQuadruples}. We conclude that \eqref{deimPAB} must hold.
\end{proof}

\section{Blaschke curvature and discretized projections} \label{BlaschkeSection}
In this section we will prove Theorem \ref{curvatureProjections}. First, we will show that if a non-concentrated set has small projection under three different maps, then the image of this set under these projections must also be non-concentrated. 

\begin{lem}\label{findingLargeNonconcentratedCartesianProduct}
Let $Q\subset[0,1]^2$ be a square, let $P\colon Q\to\RR$ be differentiable, and suppose
\begin{equation}\label{boundsOnDerivativeP}
\begin{split}
&C^{-1}\leq \inf_{Q}|\partial_x P| \leq \sup_{Q}|\partial_x P|\leq C,\\
&C^{-1}\leq \inf_{Q}|\partial_y P| \leq \sup_{Q}|\partial_y P| \leq C.
\end{split}
\end{equation}
Let $X\subset Q$ be a union of squares of the form $[m\delta,(m+1)\delta]\times[n\delta,(n+1)\delta]$, with $\mathcal{E}_{\delta}(X) = \delta^{-2\alpha}$. Suppose that for all balls $B$ of radius $r\geq\delta$ we have 
\begin{equation}\label{nonConcentrationOnBalls}
\mathcal{E}_{\delta}(X \cap B) \leq r^{2\alpha}\delta^{-2\alpha-\eta}.
\end{equation}

Suppose that
\begin{equation}\label{EHasSmallProjections}
\begin{split}
\mathcal{E}_{\delta}(\pi_x(X))& \leq \delta^{-\alpha-\eta},\\
\mathcal{E}_{\delta}(\pi_y(X))&\leq \delta^{-\alpha-\eta},\\
\mathcal{E}_{\delta}(P(X))&\leq \delta^{-\alpha-\eta}.
\end{split}
\end{equation}

Then there exists sets $A\subset \pi_x(X),\ B\subset\pi_y(X)$ that are unions of intervals of the form $[m\delta,(m+1)\delta]$, with
\begin{equation}\label{ABLargeIntersectionX}
\mathcal{E}_{\delta}(X \cap (A\times B))\geq\mathcal{E}_{\delta}(X)/2,
\end{equation}
such that for all intervals $J$ of length at least $\delta$, we have
\begin{equation}\label{ABNonConcentrated}
\begin{split}
&\mathcal{E}_{\delta}(A\cap J)\leq (C^{\frac{5}{2}} |\log\delta|^{-1}\delta^{-\frac{5}{2}\eta}) |J|^{\alpha}\delta^{-\alpha},\\
&\mathcal{E}_{\delta}(B\cap J)\leq (C^{\frac{5}{2}}|\log\delta|^{-1}\delta^{-\frac{5}{2}\eta}) |J|^{\alpha}\delta^{-\alpha}.
\end{split}
\end{equation}
\end{lem}
\begin{proof}
Let $X'\subset (\delta\ZZ)^2$ be the set of bottom-left corners of the squares in $X$. Consider the bipartite graph $G=G(\pi_x(X') \sqcup \pi_y(X'),\ X')$; two vertices $x\in \Pi_x(X')$ and $y\in \pi_y(X')$ are adjacent if $(x,y)\in X'$.

Using a standard popularity argument (see e.g. \cite[Lemma 2.8]{DG}), we can select sets $A\subset \pi_x(X')$, $B\subset \pi_y(X')$ and an edge set $E\subset X'$ so that $\#E\geq (\#X')/2$; each element of $A$ is adjacent to at least $(\#X')/(4\#(\pi_x(X')))\gtrsim \delta^{-\alpha+\eta/2}$ elements of $B$; and each element of $A$ is adjacent to at most $\#B=\delta^{-\alpha-\eta/2}$ elements of $B$. Similarly, each element of $B$ is adjacent to at least $(\#X)/(4(\#\pi_y(X')))\gtrsim \delta^{-\alpha+\eta/2}$ elements of $A$, and at most $\#A=\delta^{-\alpha-\eta/2}$ elements of $A$.

Let $J\subset$ be an interval of length $t\geq\delta$ and suppose
\[
\#(A \cap J) = K t^{\alpha} \delta^{-\alpha},
\]
for some $K\geq 1$. Then 
\begin{equation}\label{lowerBoundECapPiInvJ}
\#(E \cap \pi_x^{-1}(J)) \gtrsim K t^{\alpha}\delta^{-2\alpha+\eta}.
\end{equation}
For each $p\in E \cap \pi_x^{-1}(J)$, let $m(p)$ be the number of points $p'\in E \cap \pi_x^{-1}(J)$ with $\pi_y(p) = \pi_y(p')$. We have $1\leq m(p) \leq \#(A \cap J)$. By dyadic pigeonholing, there is a number $1\leq m \leq \#(A \cap J)$ and a set $E'\subset (E \cap \pi_x^{-1}(J))$ with $\#E' \geq |\log\delta|^{-1} \#(E \cap \Pi_x^{-1}(J))$ so that $m(p)\sim m$ for all $p\in E'$. Note that if $p\in E'$ and $\pi_y(p) = \pi_y(p')$, then $p'\in E$.

We have $\pi_y(E')\subset B$, and thus $\#(\pi_y(E'))\leq \#B\leq \delta^{-\alpha-\eta}$. On the other hand $(\#E')/m \leq \#(\pi_y(E'))$, so by \eqref{lowerBoundECapPiInvJ} we have
\[
m \gtrsim \frac{|\log\delta|^{-1} K t^{\alpha}\delta^{-2\alpha+\eta}}{\delta^{-\alpha-\eta}}\gtrsim Kt^{\alpha}|\log\delta|^{-1}\delta^{-\alpha+2\eta}.
\]

By \eqref{nonConcentrationOnBalls} we have that 
\[
\mathcal{E}_{C^2t}(E') \geq \frac{E'}{(C^2t)^{2\alpha}\delta^{-2\alpha-\eta}}\gtrsim C^{-4\alpha}|\log\delta|^{-1} K t^{-\alpha}\delta^{2\eta}.
\]

Thus we can select a $C^2t$-separated set $B'\subset B$ with 
\[
\#B' \gtrsim C^{-4\alpha}|\log\delta|^{-1} K t^{-\alpha}\delta^{2\eta},
\]
so that for each $b\in B'$ there are $\sim m$ points $p\in E'$ with $\pi_y(p) = b$; let $A_b$ be the projection of this set under $\pi_x$, i.e.  $A_b = \pi_x(E' \cap \pi_y^{-1}(b))$. Then by \eqref{boundsOnDerivativeP}, $\mathcal{E}_{\delta}(P(A_b))\gtrsim C^{-1}m$, and $P(A_b)$ is contained in an interval of length $\leq Ct$. Since the points in $B'$ are $\geq C^2t$ separated, by \eqref{boundsOnDerivativeP} we have that if $b$ and $b'$ are distinct, then the sets $P(A_b)$ and $P(A,b')$ are contained in disjoint intervals. 

Thus 
\[
\mathcal{E}_{\delta}(E') \geq \sum_{b\in B'} \mathcal{E}_{\delta}{E' \cap \pi_y^{-1}(b)}\geq C^{-1} (\#B') m
\]

We conclude that
\begin{equation}
\begin{split}
\mathcal{E}_{\delta}(P(X)) & \geq \mathcal{E}_{\delta}(E')\\
& \geq C^{-1} |B'| m\\
& \gtrsim  C^{4\alpha+1}|\log\delta|^{-1} K t^{-\alpha}\delta^{2\eta} Kt^{\alpha}|\log\delta|^{-1}\delta^{-\alpha+2\eta}\\
& = K^2  C^5|\log\delta|^{-2}   \delta^{-\alpha+4\eta}.
\end{split}
\end{equation}
Comparing this with \eqref{EHasSmallProjections}, we conclude that
\[
K \leq C^{\frac{5}{2}}|\log\delta|^{-1}   \delta^{-\frac{5}{2}\eta}.
\]
This establishes the first inequality in \eqref{ABNonConcentrated}. An identical argument establishes the second inequality. Our sets $A$ and $B$ are currently subsets of $(\delta\ZZ)$ rather than unions of intervals of the form $[m\delta,(m+1)\delta]$. To fix this, replace $A$ with $A + [0,\delta]$ and replace $B$ with $B + [0,\delta]$. 
\end{proof}

With this lemma, we are now ready to begin the proof of Theorem \ref{curvatureProjections}. For the reader's convenience we will restate the theorem here

\begin{curvatureProjectionsThm}
Let $0< \alpha< 1$, let $U\subset\RR^2$ be a connected open set, and let $K\subset U$ be compact. Let $\phi_1,\phi_2,\phi_3\colon U\to\RR$ be analytic functions whose gradients are pairwise linearly independent at each point of $U$. Then either the Blaschke curvature form \eqref{BlaschkeCurvatureForm} vanishes identically on $U$, or there exists $\eps = \eps(\alpha)>0$ and $\eta = \eta(\alpha,\phi_1,\phi_2,\phi_3)>0$ so that the following is true for all $\delta>0$ sufficiently small.

Let $X\subset[0,1]^2$ be a set with $\mathcal{E}_{\delta}(X)\geq \delta^{-2\alpha}$, and suppose that for all balls $B$ of diameter $r\geq\delta$, $X$ satisfies the non-concentration condition
\begin{equation*}\tag{\ref{nonConcentrationCondX}}
\begin{split}
\mathcal{E}_{\delta}(X \cap B) &\le  r^\alpha\delta^{-2\alpha-\eta}.
\end{split}
\end{equation*}
Then for at least one index $i\in \{1,2,3\}$ we have
\begin{equation*}\tag{\ref{atLeastOneBigProjection}}
\mathcal{E}_{\delta}(\phi_i(X))\geq \delta^{-\alpha-\eps}.
\end{equation*}
\end{curvatureProjectionsThm}
\begin{proof}
Suppose that the Blaschke curvature form does not vanish identically on $U$, and that \eqref{atLeastOneBigProjection} is false for $i=1,2$. If $\eta=\eta(\alpha,\phi_1,\phi_2,\phi_3)$ and $\eps=\eps(\alpha)$ are selected sufficiently small, we will show that \eqref{atLeastOneBigProjection} holds for $i=3$. 

For each point $p\in K$, there is a neighborhood $U_p\subset U$ and an analytic change of coordinates so that $\phi_1(x,y) = x$ and $\phi_2(x,y) = y$. For each set $U_p$, let $S_p\subset U_p$ be a closed square (in the straightened coordinates) whose interior $S_p^{\circ}$ contains $p$. Since the open sets $\{S_p^{\circ}\}$ cover $K$, by compactness we can select a finite sub-cover $\{S_p^{\circ}\}_{p\in\mathcal{P}}$. 

By pigeonholing, there exists $p_0\in\mathcal{P}$ so that $\mathcal{E}_{\delta}(S_{p_0}\cap X)\geq \mathcal{E}_{\delta}(X)/|\mathcal{P}|$. Let $U' = U_{p_0}$ and $S = S_{p_0}$. Applying the analytic change of coordinates described above, we can suppose that $\phi_1(x,y) = x$ and $\phi_2(x,y) = y$ on the square $S$. Define $P$ to be the function $\phi_3$ after this change of coordinates. We have that $P$ is analytic on $U'\supset S$.



Since $S$ has diameter $\sim 1$, after translation and isotropic rescaling, we can suppose that $S = [0,1]^2$. This rescaling will introduce a harmless multiplicative factor (depending only $\phi_1,\phi_2,\phi_3$) into our final estimate. Define $X'$ to be the image of $S\cap X$ under this transformation and define $X_{\delta}$ to be the union of $\delta$-squares of the form $[m\delta,(m+1)\delta]\times [n\delta,(n+1)\delta]$ that intersect $X'$. Abusing notation slightly, we will continue to use $P\colon [0,1]^2\to\RR$ to refer to the function $\phi_3$ in these (straightened, rescaled) coordinates. 

Since the gradients of $\phi_1,\phi_2,\phi_3$ are pairwise linearly independent at each point of $U$, there is a constant $C$ (depending on $\phi_1,\phi_2,\phi_3$ but independent of $\delta$ and $X$) so that
\begin{equation}\label{sizeOfGradients}
C^{-1}\leq |\partial_x P|,\ |\partial_yP| \leq C \quad\textrm{on}\ [0,1]^2.
\end{equation}

Since $\mathcal{E}_{\delta}(\phi_3(X))\gtrsim \mathcal{E}_{\delta}(P(X_{\delta}))$ (with an implicit constant that depends only on $\phi_1,\phi_2,\phi_3$), in order to prove that \eqref{atLeastOneBigProjection} holds for $i=3$ (and $\delta>0$ sufficiently small), we must show that.
\begin{equation}\label{needToShowPXpLarge}
\mathcal{E}_{\delta}(P(X_{\delta}))\geq  \delta^{-\alpha-\eps}.
\end{equation}
Suppose \eqref{needToShowPXpLarge} fails. We will obtain a contradiction if $\eps = \eps(\alpha)$ and $\eta = \eta(\alpha,P)=\eta(\alpha,\phi_1,\phi_2,\phi_3)$ is selected sufficiently small.

Since \eqref{atLeastOneBigProjection} fails for $i=1,2$, we have that 
\begin{equation}
\begin{split}
&\mathcal{E}_{\delta}(\pi_x (X_{\delta}))\lesssim \delta^{-\alpha-\eps},\\
&\mathcal{E}_{\delta}(\pi_y (X_{\delta}))\lesssim \delta^{-\alpha-\eps}.
\end{split}
\end{equation}
Since we are supposing that \eqref{needToShowPXpLarge} fails, we can apply Lemma \ref{findingLargeNonconcentratedCartesianProduct} (we can apply this lemma provided we choose $\eta\leq\eps$) to obtain sets $A,B\subset[0,1]$ that are unions of intervals of the form $[m\delta,(m+1)\delta]$ that satisfy

\begin{equation}\label{sizeOfAdeltaBdelta}
\begin{split}
&\mathcal{E}_{\delta}(A) \lesssim \delta^{-\alpha-\eps},\\
&\mathcal{E}_{\delta}(B) \lesssim \delta^{-\alpha-\eps},
\end{split}
\end{equation}
and for all intervals $J$ of length at least $\delta$, 
\begin{equation}\label{nonConcAB}
\begin{split}
&\mathcal{E}_{\delta}(A\cap J)\lesssim  |J|^{\alpha}\delta^{-\alpha-3\eps},\\
&\mathcal{E}_{\delta}(B\cap J)\lesssim |J|^{\alpha}\delta^{-\alpha-3\eps},
\end{split}
\end{equation}
and 
\begin{equation}\label{XDeltaLargeIntersectionATimesB}
\mathcal{E}_{\delta}\big( X_{\delta} \cap (A\times B)\big) \gtrsim \mathcal{E}_{\delta}(X_{\delta})\gtrsim \delta^{-2\alpha},
\end{equation}
i.e.
\begin{equation}\label{XDeltaLargeIntersectionATimesBV2}
\mathcal{E}_{\delta}\big( X_{\delta} \cap (A\times B)\big) \gtrsim \delta^{2\eps}\mathcal{E}_{\delta}(A\times B).
\end{equation}


Since the Blaschke curvature form \eqref{BlaschkeCurvatureForm} is invariant under change of coordinates, we have that $\partial_{xy}\log\Big(\frac{\partial_x P}{\partial_y P}\Big)$ does not vanish identically on $[0,1]^2$. By \eqref{relateKpLog} we have that $K_P(x,y)$ does not vanish identically on $[0,1]^2$. 

The remainder of the argument is similar to the proof of Theorem \ref{main-entropy-growth} from Section \ref{mainThmProofsSectionOne}. Let $C$ be the constant obtained by applying Lemma \ref{cubesWhereAnalyticFLarge} to the function $K_P$. We will apply this lemma with $w \sim C\eps\eta/\alpha$ ($\eta>0$ will be chosen small enough that $w<\alpha/2$), and let $\mathcal{Q}$ be the resulting collection of cubes. With this choice of $w$, we have
\begin{equation*}
\mathcal{E}_{\delta}\Big( (A \times B) \backslash \bigcup_{Q\in\mathcal{Q}}Q\Big)\leq \delta^{3\eps}\mathcal{E}_{\delta}(A\times B).
\end{equation*}
Thus by \eqref{XDeltaLargeIntersectionATimesBV2}, if $\delta>0$ is selected sufficiently small then
\begin{equation*}
\begin{split}
\mathcal{E}_{\delta}\Big( X_{\delta} \cap (A \times B ) \cap  \bigcup_{Q\in\mathcal{Q}}Q\Big) &\geq
\mathcal{E}_{\delta}\big( X_{\delta} \cap (A\times B)\big) - \mathcal{E}_{\delta}\Big( (A \times B) \backslash \bigcup_{Q\in\mathcal{Q}}Q\Big)\\
&\geq\frac{1}{2}\mathcal{E}_{\delta}(X_{\delta}).
\end{split}
\end{equation*}
By pigeonholing, there is a cube $Q\in\mathcal{Q}$ so that
\begin{equation}\label{lotsOfMassInCube}
\mathcal{E}_{\delta}(X_{\delta}\cap (A \times B )\cap Q) \gtrsim \delta^{4w/\alpha}\mathcal{E}_{\delta}(A \times B )\gtrsim \delta^{4C\eta/\alpha^2}\mathcal{E}_{\delta}(X_{\delta})\gtrsim \delta^{4C\eps\eta/\alpha^2-2\alpha}.
\end{equation}
Let $A_1\times B_1 = (A\times B)\cap Q$.  If $\eta = \eta(\alpha,C)=\eta(\alpha,\phi_1,\phi_2,\phi_3)$ is selected sufficiently small, then by  Lemma~\ref{countingNumberQuadruplesNablePWedgeP}, there exists $\eps' = \eps'(\alpha)>0$ so that 
\[
{\mathcal E}_\delta\big(\{ ( x,  x',  y,  y')\in A_1^2\times B_1^2\colon P(x,y)=P(x',y')\}\big)\lesssim \delta^{-3\alpha+\eps'}. 
\]
But this implies
\begin{equation}\label{upperBoundOnQuadruples}
{\mathcal E}_\delta\big(\{ ( x,y,x',y')\in \big(X_{\delta}\cap (A\times B) \cap Q\big)^2 \colon P(x,y)=P(x',y')\}\big)\lesssim \delta^{-3\alpha+\eps'}.
\end{equation}
Applying Lemma \ref{CSLem} to $Q\cap X'$ with $c = \delta^w$ and using \eqref{lotsOfMassInCube} and \eqref{upperBoundOnQuadruples}, we conclude that
\[
\mathcal{E}_{\delta}(P(X_{\delta}))\geq \mathcal{E}_{\delta}(P(X_{\delta}\cap (A\times B) ))\gtrsim 
\frac{\big(\delta^{4C\eps\eta/\alpha^2-2\alpha}\big)^2}{\delta^{-3\alpha+\eps'}} = \delta^{8C\eps\eta/\alpha^2-\alpha-\eps'}.
\]
Since we are supposing that \eqref{needToShowPXpLarge} fails, we have 
\[
\delta^{8C\eps\eta/\alpha^2-\alpha-\eps'} \lesssim  \delta^{-\alpha-\eps}.
\]
For $\delta$ sufficiently small, this implies $8C\eps\eta/\alpha^2+\eps-\eps' > 0$. Selecting $\eta \leq \alpha^2/(8C)$ and $\eps<2\eps'$ we obtain a contradiction. We conclude that \eqref{needToShowPXpLarge} holds, which completes the proof.  
\end{proof}

\end{document}